\documentclass{article}
\usepackage{color}
\usepackage{amsfonts}
\usepackage{tikz}
\usepackage{amssymb}
\usepackage{amsmath}
\usepackage{circuitikz}
\usepackage{amsthm}
\usepackage[mathscr]{euscript}
\usepackage{hyperref}
\usepackage{geometry}
\usepackage{comment}
\usepackage{authblk}
\usepackage[T1]{fontenc}
\usepackage[utf8]{inputenc}

\newtheorem{theorem}{Theorem}[section]
\newtheorem{corollary}{Corollary}[section]

\newtheorem{lemma}{Lemma}[section]
\newtheorem{proposition}{Proposition}[section]
\newtheorem{remark}{Remark}[section]

\numberwithin{equation}{section}

\newcommand{\m}{\mathbb}
\newcommand{\ml}{\mathcal}
\newcommand{\p}{\partial}
\newcommand{\pprec}{\prec\!\!\!\prec}

\makeatletter
\newcommand\RSloop{\@ifnextchar\bgroup\RSloopa\RSloopb}
\makeatother
\newcommand\RSloopa[1]{\bgroup\RSloop#1\relax\egroup\RSloop}
\newcommand\RSloopb[1]%
{\ifx\relax#1%
	\else
	\ifcsname RS:#1\endcsname
	\csname RS:#1\endcsname
	\else
	\GenericError{(RS)}{RS Error: operator #1 undefined}{}{}%
	\fi
	\expandafter\RSloop
	\fi
}
\newcommand\X{0}
\newcommand\RS[1]%
{\begin{tikzpicture}
		[every node/.style=
		{circle,draw,fill,minimum size=1.5pt,inner sep=0pt,outer sep=0pt},
		line cap=round
		]
		\coordinate(\X) at (0,0);
		\RSloop{#1}\relax

	\end{tikzpicture}
}
\makeatother
\newcommand\RSdef[1]{\expandafter\def\csname RS:#1\endcsname}
\newlength\RSu
\RSdef{i}{\draw (\X) -- +(90:\RSu) node{};}
\RSdef{l}{\draw (\X) -- +(135:\RSu) node{};}
\RSdef{r}{\draw (\X) -- +(45:\RSu) node{};}
\RSdef{I}{\draw (\X) -- +(90:\RSu) coordinate(\X I);\edef\X{\X I}}
\RSdef{L}{\draw (\X) -- +(135:\RSu) coordinate(\X L);\edef\X{\X L}}
\RSdef{R}{\draw (\X) -- +(45:\RSu) coordinate(\X R);\edef\X{\X R}}
\RSdef{n}{\draw node{};}
\RSdef{c}{\draw node[fill=none]{};}

\begin{document}
	\author[1,\footnote{Email: zhangshlin9@mail2.sysu.edu.cn}]{\scshape  Shuolin Zhang}
	\author[1,\footnote{Email: luozhn7@mail.sysu.edu.cn }]{\scshape  Zhaonan Luo}
	\author[1,\footnote{Email: mcsyzy@mail.sysu.edu.cn}]{\scshape  Zhaoyang Yin}
	\affil[1]{\slshape School of Science, Sun Yat-sen University, ShenZhen 518107, PR China}
	\title{Well-posedness and regularity for the Fractional Rough Burgers equation in Sobolev spaces with an application}
	\date{}
    \maketitle	
    
	\begin{abstract}
	In this paper, we study the well-posedness of Fractional Rough Burgers equation driven by space-time white noise in $H^s(\m T)$ space. For the higher dissipation $\gamma\in(\frac{4}{3},2]$, we establish local well-posedness. Global well-posedness is further obtained when $\gamma$ is restricted to the interval $(\frac{8}{5}, 2]$.  For the lower dissipation $\gamma\in(\frac{5}{4},\frac{4}{3}]$, we establish the para-controlled solution in $\ml C^{s}(\m T)\cap H^s(\m T)$.\\
		\vspace{0.5cm}
	\end{abstract}
	\vspace{0.5cm}
	\tableofcontents
	\maketitle
	
	\section{Introduction}

	In this paper, we will consider Fractional Rough Burgers (FRB) equation with weak dissipation in $H^s(\mathbb{T})$
	\begin{equation}\label{eq;Burgers;white noise}
		\p_t u-\Lambda^\gamma u=\p_x(u^2)+\xi,\quad u(0,x)=u_0.
	\end{equation}
	Here  $\mathbb{T}$ is some torus satisfied $\int_{\mathbb{T}}e^{ikx}=0$ for any $k\in \mathbb{Z}/\{0\}$(for example $\m R/2\pi \m Z$) and $\Lambda^\gamma$ is a Fourier multiplier satisfied $\ml F(\Lambda^\gamma f)(k)=-\vert k\vert^\gamma\ml Ff$. $\xi$ is a centered Gaussian space-time random distribution with covariance
	\begin{equation}\label{eq;white noise}
		\m E[\xi(\omega,t,x)\xi(\omega,s,y)]=\delta_0(t-s)\delta_0(x-y),\quad t,s\geq0,\quad x,y\in\mathbb{T},
	\end{equation}
	where $\delta_0(x)$ is a Dirac measure satisfied $\delta_0(x)=0$ for $x\neq0$ and $\int_{\m T}\mathrm{d}\delta_0(x)=1$. 
	\par The stochastic Burgers equation (SBE) on the one dimensional torus $\m T$ is the SPDE
	\begin{equation}\label{eq;Burergs;derivative noise}
		\partial_t u - \Delta u = \partial_x (u^2) + \partial_x \xi
	\end{equation}
	such equation is important, as its solution is the derivative of the solution of the Kardar–Parisi–Zhang (KPZ) equation, which is believed to capture the macroscopic behavior of a large class of surface growth phenomena. In a weak sense, Bertini and Giacomin solve it in \cite{bertini1997stochastic}  by the classical Cole–Hopf transformation, which reduces it to a linear equation with multiplicative half‑white noise:
	\begin{equation}
		dZ_t = \frac12 \partial_\xi^2 Z_t(\xi) \, dt + Z_t(\xi) \, dW_t(\xi)
	\end{equation}
	though uniqueness is lacking in this formulation. Uniqueness remained a difficult open problem until Gubinelli and collaborators studied the fractional stochastic Burgers equation
	\begin{equation}\label{eq;frac burgers}
		\partial_t u - \Lambda^\theta u = \partial_x (u^2) + \Lambda^{\frac{\theta}{2}} \xi
	\end{equation}
	for \(\theta > \frac54\) via noise regularization in \cite{gubinelli2013regularization} (with Jara), and later returned to the case \(\theta = 1\) in\cite{gubinelli2018energy} (with Perkowski).
	\par A new and advanced mild formulation, which provides a pathwise notion of solution, comes from the work of Martin Hairer on the KPZ equation with a renormalization:
	\begin{equation}
		\partial_t h = \partial_x^2 h + \lambda (\partial_x h)^2 - \infty
	\end{equation}
	Hairer used the idea of controlled rough paths to first describe the random growth of an interface and gave detailed regularity results for the solution. In later work, Hairer built the theory of regularity structures with the aim of providing a more general and versatile notion of regularity. Around the same period, Gubinelli, Imkeller and Perkowski developed the theory of paracontrolled distributions in \cite{gubinelli2015paracontrolled}, a tool that gives meaning to and solves a large class of singular SPDEs. This approach attracted widespread attention because it combines harmonic analysis with rough paths and adapts to various structures such as RDE case \cite{gubinelli2015paracontrolled}, parabolic case \cite{gubinelli2015paracontrolled,gubinelli2017kpz,catellier2018paracontrolled} and dispersive case \cite{gubinelli2023paracontrolled,deng2024invariant,deng2022random}, even though it has certain restrictions compared to the regularity structure framework. In particularly, Gubinelli and Perkowski give a pathwise proof of global existence of solution of \eqref{eq;frac burgers} for $\theta=1$ in \cite{gubinelli2017kpz}. 
	\par In addition, there exist some results for the fractional Burgers equation driven by noise in a probabilistic framework. For instance, Wang, Zeng and Guo established local and global existence as well as uniqueness for the stochastic Burgers equation driven by cylindrical fractional Brownian motion with Hurst parameter $H > \frac{1}{4}$. Jiang, Wei and Zhou proved the existence and uniqueness of solutions for a class of stochastic generalized Burgers equations driven by multiparameter fractional noises with Hurst parameter $H > \frac{1}{2}$.
	\par In this paper, we study \eqref{eq;Burgers;white noise} by combining the regularity techniques introduced in \cite{gubinelli2015paracontrolled} with the paraproduct machinery and fixed point argument introduced in \cite{bahouri2011fourier}. Instead of considering weak solutions for singular PDEs, our solution is constructed in the mild sense and pathwise, as there is no difficulty in defining the nonlinearity. In our setting, we treat the noise as a forcing term, which allows us to establish local and global well‑posedness following the approach presented in Section 5 of \cite{bahouri2011fourier}—a fundamental method for solving the 2D Navier–Stokes equations. Our method is easy to learn and adapts naturally to deterministic equations solved by fixed‑point arguments. 
	\par Under our approach, we obtain results that clarify how the dissipation parameter $\gamma$ and the initial data $u_0$ affect the well‑posedness and regularity of solutions. In the deterministic case, the solution always retains the regularity of $u_0$; this property, however, breaks down for rough PDEs due to the low regularity of space‑time white noise. A key observation is that even if the solution $u$ does not attain sufficiently high regularity, it can be decomposed into several parts, showing that $u$ always consists of a rough component and a smooth component. In this paper, we consider equation \eqref{eq;Burgers;white noise} under two regimes of lower dissipation:
	\begin{itemize}
		\item{For $\frac{4}{3}<\gamma\leq2$, we assume initial data $u_0 \in H^{\frac{1}{2}+}(\mathbb{T})$. Theorem \ref{th;local result} shows that $u$ can be decomposed as $X+v$ where $X$ is the stochastic convolution and $v$ belongs to $H^{\frac12+}$. }
		\item{For $\frac{5}{4} < \gamma \leq \frac{4}{3}$, we work with initial data in the intersection $u_0 \in H^{\frac{1}{2}+}(\mathbb{T}) \cap \mathcal{C}^{\frac{1}{2}+}(\mathbb{T}) = W^{\frac{1}{2}+}(\mathbb{T})$. Theorem \ref{th;paracontrolled} shows that $u$ admits a paracontrolled decomposition, with a remaining higher-order term $u^\sharp \in W^{\frac{1}{2}+}$.}
	\end{itemize}
	\par In the range $\gamma \in (\frac85, 2]$, we provide a proof of global well-posedness when $u_0$ has zero mean. This proof relies on an energy estimate and a blow-up criterion, which together constitute an innovation of our method. Such an argument is useful for solving conservation-type equations, including the Navier–Stokes equation, the Schrödinger equation, and a large class of shallow water equations. The reader may refer to \cite{bahouri2011fourier,taoNonlinearDispersiveEquations2006,gui2011cauchy,gui2015global}. Although the random forcing breaks the conservative structure of the equation, an energy estimate can still be established in $L_T^\infty L^2 \cap L_T^2 H^{\frac{\gamma}{2}}$, with an upper bound involving some globally defined Gaussian process in the non‑singular setting (see Proposition \ref{prop;est;energy}). Moreover, the blow-up criterion in $L_T^4 \dot H^{\frac{\gamma}{4}}$ allows us to extend the local solution to all positive times.
	\par It is  also need to particularly noteworthy that the first author and co-authors have previously studied the randomized initial data problems for both the DP equation and Burgers equation
	\begin{equation}
		\p_t u-\Lambda^\gamma u=uu_x,\quad u_0^\omega=\sum_{n\in \m Z} h_n(\omega) \varphi(\Lambda-n)u_0.
	\end{equation}
	in \cite{chen2024probabilistic} for $\gamma\in(\frac{5}{4},\frac{3}{2})$, which is an another hot topic in the field of SPDEs. We omit detailed discussion here and refer interested readers to \cite{bourgainInvariantMeasuresThe2Ddefocusing1996,burqRandomDataCauchy2008a,burqRandomDataCauchy2008c,chenCauchyProblemHartree2015,benyiProbabilisticCauchyTheory2015,benyiHigherOrderExpansions2019,deng2022random,deng2024probabilistic}. Therefore, this paper can be viewed as a natural continuation of that earlier work under additional noise driving.
	\par  The primary objection of this paper is to establish fundamental results for the FRB equation in Sobolev spaces $H^s$ and extend analogous conclusions to the DP equations. We will address the potential difficulties associated with these problems, as well as the methods we will employ, by organizing the discussion into three main sections.
	\subsection{Probabilistic setting}
	The fact that $\xi \in \mathcal{C}^{-\frac{1}{2}-}(\mathbb{R}_+ \times \mathbb{T})$ introduces significant analytical difficulties for direct approaches. One strategy is to write equation as mild form
	\begin{equation}\label{eq;mild solution}
		u=P(t)u_0+\int_0^tP(t-s)\p_x(u^2)\mathrm{d}s+X(t).
	\end{equation}
	Here $P(t)=e^{t\Lambda^\gamma}$ is the fractional heat flow, and $X(t)=\int_0^t P(t-s)\xi(s)\mathrm{d}s$ denote the solution of linear equation
	\begin{equation}\label{eq;linear evolution;noise}
		\p_t X(t)-\Lambda^\gamma X(t)=\xi,\quad X(0)=0.
	\end{equation}
	Noting that $X(t)\in C_T\ml C^\alpha(\mathbb{T})$ holds almost surely for any $\alpha<\frac{\gamma}{2}-\frac{1}{2}$, Such regularity has played a key role in establishing well-posedness for equations modeling thermally driven dissipative structures, notably the KPZ equation \cite{hairer2013solving} and the Burgers equation \cite{gubinelli2015paracontrolled,gubinelli2017kpz}. In this paper, we extend analogous regularity results to the Sobolev spaces $H^s$ and apply them to prove well-posedness for the class of equations considered.
	\par  Since the admissible range of $\alpha$ depends on the dissipation parameter $\gamma$, this paper presents two regimes with distinct analytical treatments.
	\begin{itemize}
		\item{Higher regularity regime ($\alpha \in (\frac{1}{6},\frac{1}{2})$): in this case, well-posedness can be established directly in the Sobolev space $H^\alpha(\mathbb{T})$.}
		\item{Lower dissipation regime ($\alpha \in (\frac{1}{8},\frac{1}{6})$, see Remark \ref{rmk;no endpoint}): here the analysis requires working in $W^{\alpha}(\mathbb{T})$.}
	\end{itemize}
	A common feature in both regimes is the isolation of a higher-regularity component in the solution—more precisely, a term belonging to $H^{\frac{1}{2}+}(\mathbb{T})$. This decomposition is made possible for lower dissipation levels $\gamma \in (\frac{5}{4},\frac{4}{3}]$ by means of the paracontrolled calculus, a powerful tool introduced for studying rough and singular stochastic partial differential equations \cite{gubinelli2015paracontrolled,gubinelli2017kpz,gubinelli2023paracontrolled,catellier2018paracontrolled}.
	\par A more in-depth consideration is that how lower the dissipation be? Let us consider FRB driven by $\vert D\vert^\beta \xi$
	\begin{equation}
		u_t-\Lambda^\gamma u=\p_x(u^2)+\vert D\vert^\beta\xi,\quad u(0,x)=u_0.
	\end{equation}
	Using the similar method in \cite{deng2022random,deng2024probabilistic}, we propose a conjecture regarding the relationship between $\gamma$ and $\beta$ 
	\begin{equation}
		\gamma\geq\frac23\beta+1
	\end{equation}
	in final section.  Similar methods have been used in \cite{deng2024probabilistic} to study probabilistic critical problems for the heat equation, wave equation, and Schrödinger equation.
	\subsection{Regularity analysis}
	After resolving the regularity issues for Gaussian processes, the next challenge lies in the study of well-posedness. Consider the Burgers equation with addition term $f$
	\begin{equation}\label{eq;Burgers;addition}
		\p_t v-\Lambda^\gamma v=\p_x(v^2)+f,\quad v(0,x)=v_0,
	\end{equation}
	where $f$ may depend on $u$. It takes the different structure of solution in two different ranges $\gamma\in(\frac{4}{3},2]$ and $\gamma\in (\frac{5}{4},\frac{4}{3}]$.
	\par More precisely, let $f(v)=2\p_x(vX)+\p_x(X^2)$, where $X$ satisfied \eqref{eq;linear evolution;noise}. Then $f(v)$ gains at most $\alpha$-order Sobolev regularity when $v$ belongs to $H^{\frac{1}{2}+}$. It derive the relation between $\alpha$ and $\gamma$ which satisfied $\alpha+\gamma\geq \frac{3}{2}$, since $v$ can gain $\alpha+\gamma-1$ at most. 
	\par For the lower $\alpha$ and $\gamma$, our objective is to similarly isolate a higher-regularity remainder term ($H^{\frac{1}{2}+}$) in the solution’s structure. As the method used in the highly dissipative regime is no longer applicable here, a more refined analysis is required. Using the Littlewood–Paley decomposition (see \cite{bahouri2011fourier}), we rewrite $vX$ and $X^2$ as 
	\begin{equation}\label{eq;decomposition; VX XX}
		vX=T_vX+R(v,X)+T_Xv,\quad X^2=2T_XX+R(X,X).
	\end{equation} 
	The terms $T_Xv$, $R(v,X)$ and $R(X,X)$ inherit higher regularity due to the fact that $X$  possesses $\alpha$-order H\"oler regularity. The main difficulty in the small-$\alpha$ case lies in handling the remaining terms $T_v X$ and $2T_XX$. We address this difficulty using the paracontrolled method (see Proposition \ref{prop;local result;1} for details), which ultimately yields the condition $2\alpha-1+\gamma\geq\frac{3}{2}$.
	\subsection{Shallow water equation}
	Next, we introduce the Degasperis-Procesi(DP) equation, which is one type of shallow water equation, has the following form
	\begin{equation}\label{eq;DP equation;no dissipative}
		m_t+um_x+3u_xm=0,\quad m(t,x)=(1-\partial_x^2)u(t,x),\quad u(0,x)=u_0(x).
	\end{equation}
	Then we can define the fluid velocity $u(t,x):=g\ast m:=(1-\p_x^2)^{-1}m$ and rewrite the DP equation as
	\begin{equation}\label{eq;DP equation;general}
		u_t+\frac{1}{2}\p_x(u^2)+\frac{3}{2}(1-\partial_x^2)^{-1}\partial_x(u^2)=0,\quad u(0,x)=u_0(x).
	\end{equation}
	The following citations \cite{degasperis1999asymptotic,degasperis2002new} provide physical background for this class of equations; a more systematic discussion can be found in \cite{ai2010global}. Regarding the mathematical theory, we note in particular that local well-posedness of \eqref{eq;DP equation;no dissipative} for initial data in $H^s$ with $s>\frac{3}{2}$ was established in \cite{liu2006global}, together with a precise blow-up criterion and blow-up results in \cite{liu2007blow}.
	\par In this paper, we also consider the nonlocal Rough DP equation with space-time noise $\xi$
	\begin{equation}\label{eq;DP;Burgers-like}
		u_t-\Lambda^\gamma u+\frac{1}{2}\p_x(u^2)+\frac{3}{2}(1-\partial_x^2)^{-1}\partial_x(u^2)=\xi,\quad u(0,x)=u_0(x),
	\end{equation}
	as an application of our result for Rough Burgers equation. Such method is also suitable for other type of shallow water equation with fractional dissipation.
	\subsection{Main result}
	\par We now present the main conclusions. 
	\begin{theorem}\label{th;local result}
		Let $u_0(x)\in H^{\frac{1}{2}+\delta}(\m T)$ for  $0<\delta<\frac{3\gamma}{8}-\frac{1}{2}$, $\gamma\in (\frac{4}{3},2]$, then there exists a postive $T^*$ and a solution $u(t,x)$ of equation \eqref{eq;Burgers;white noise}  satisfying 
		$$u(t,x)=X(t,x)+v(t,x),$$
		where $v(t,x)\in C_{T^*}H^{\frac{1}{2}+\delta}(\m T)$ a.s. and $X(t,x)$ satisfies \eqref{eq;linear evolution;noise} with $X(t,x)\in C_{T^*}H^\alpha(\m T)$ a.s for $\alpha\in(\frac{1}{6},\frac{\gamma}{2}-\frac{1}{2})$. Moreover we have $u\in C_{T^*}H^{\alpha}(\m T)$ a.s..
	\end{theorem}
	\begin{remark}
		`a.s.' represents almost surely. We say event $A$ is almost surely if and only if $\m P(A)=1$, where $\m P$ represents the probability associated with $(\Omega,\ml F,\m P)$. We will omit this description in most situation, except we need distinct such meaning.  
	\end{remark}
	\begin{remark}
		Recall the equation in \cite{gubinelli2015paracontrolled}
		\begin{equation}
			\p_tu-\Lambda^\gamma u=G(u)\p_xu+\xi,\quad u(0,x)=u_0,
		\end{equation}
		where $G\in C_b^3(\mathbb{R}^d,L(\mathbb{R}^d,\mathbb{R}^d))$. The main difference between our result and \cite{gubinelli2015paracontrolled} lies in the fact that our solution can back to the original equation \eqref{eq;Burgers;white noise}. Since $G(u)\p_xu$ is ill-defined when $\alpha\in(\frac{1}{3},\frac{\gamma}{2}-\frac{1}{2})$, they actually study the ``generalized" solution, which is the limit in probability sense of $u^\epsilon$ satisfying the following mollified equation
		\begin{equation}
			\p_t u^\epsilon-\Lambda^\gamma u^\epsilon=G(u^\epsilon)\p_xu^\epsilon+\xi^\epsilon,\quad u(0,x)=u_0.
		\end{equation}
		Here $\xi^\epsilon=\epsilon^{-1}\psi(\epsilon\cdot)\ast \xi$ for some $\psi\in \ml S$ with $\int\psi \mathrm{d}t=1$. Similar approaches are frequently used for some singular SPDEs in \cite{catellier2018paracontrolled,hairer2013solving,gubinelli2017kpz}. 
	\end{remark}
	When $\gamma>\frac{3}{2}$, we can get the global result as follows.
	\begin{theorem}\label{th;global result}
		Let $\gamma\in]\frac{8}{5},2]$, $u_0\in L^2 $ and $\int_{\m T}u_0(x)\mathrm{d}x=0$, then Rough Burgers equation \eqref{eq;Burgers;white noise} has a global mild solution $u=v+X$ where $X$ satisfies \eqref{eq;linear evolution;noise} and $v$ belongs to $C([0,T];L^2(\mathbb{T}))\cap L^2([0,T];H^\frac{\gamma}{2}(\mathbb{T}))$ for any $T\geq0$. 
	\end{theorem}
	\begin{remark}\label{rem;result;DP}
		Theorem \ref{th;local result} and \ref{th;global result} are also valid for Rough DP equation. In fact, the nonlinear term $\frac{3}{2}(1-\partial_x^2)^{-1}\partial_x(u^2)$ in the DP equation does not introduce additional difficulties for the local well-posedness problem. For global results, we need to utilize the special structure of the DP equation to derive energy estimates as Proposition \ref{le;est;energy}.
	\end{remark}
	\par For the lower dissipation $\gamma\in(\frac{5}{4},\frac{4}{3}]$, let $\ml Q$ be the linear evolution of $\p_x X$ such that
	\begin{equation}\label{eq;linear evolution;DX}
		\p_t \ml Q(t)-\Lambda^\gamma \ml Q(t)=\p_x X(t),\quad \ml Q(0)=0.
	\end{equation}
	Let $(u,u',u^\sharp)$ be a triple satisfying
	\begin{equation}
		u=X+u'\pprec \ml Q+u^\sharp,
	\end{equation}
	where $u'\pprec \ml Q$ is the temporal mollified version of $T_{u'}\ml Q$(See \eqref{eq;para-product;modify}). $u^\sharp$ is the higher regularity term.
	
	\begin{theorem}\label{th;paracontrolled}
		Let $u_0\in \ml C^{\frac{1}{2}+\delta}\cap H^{\frac{1}{2}+\delta}=W^{\frac{1}{2}+\delta}$ for  $0<\delta<\frac{\gamma}{2}-\frac{5}{8}$, $\gamma\in(\frac{5}{4},\frac{4}{3}]$. Define $\ml L=\p_t-\Lambda^\gamma$, then there exists a positive $T^*$ and a unique triple $(u,u',u^\sharp)\in C_TW^{\alpha}\times C_TW^{\frac{1}{8}+\delta}\times C_TW^{\frac{1}{2}+\delta}$ satisfying
		\begin{equation}\label{eq;paracontrolled eqaution}
			\begin{aligned}
				u=&X+u'\pprec\ml Q+u^\sharp,
			\end{aligned}
		\end{equation}
		where $X(t)$ satisfies \eqref{eq;linear evolution;noise} with $X(t,x)\in C_{T^*}W^\alpha(\m T)$ a.s for $\alpha\in(\frac{1}{8},\frac{\gamma}{2}-\frac{1}{2})$. Moreover, $u$ is a solution of \eqref{eq;Burgers;white noise}.
	\end{theorem}
	\begin{remark}
		$\prec$ and $\pprec $ are referred to as the para-product and modified para-product, which will be introduced in Lemma \ref{le;est;paraproduct} and \eqref{eq;para-product;modify}. Let $v:=u'\pprec \ml Q$, then $v$ satisfies the equation
		\begin{align*}
			\p_tv-\Lambda^\gamma v='regular\; term'+u'\prec \p_x X.
		\end{align*}
		This term will appear in the equation for $u^\sharp$. Our strategy is to use a specific $u'$ with $-\ml L(u'\pprec\ml Q)$ to cancel out this term to make sure that $u^\sharp$ can gain higher regularity (See Proposition \ref{prop;local result;1}). We will show it in Section \ref{sec;Para-controlled}.
	\end{remark}
	\begin{remark}
		In fact, we can't prove such result with only $H^\alpha$ space under our setting. Let $(u',\ml Q,u^\sharp)\in (\ml C^\alpha\times \ml C^{\beta}\times H^{\eta})$, then $\alpha+\beta=\eta$ is always valid. Since $\ml Q$ can gain $(\alpha-1+\gamma)$-order Sobolev regularity, $u^\sharp$ gains $(2\alpha-1+\gamma)$-order Sobolev regularity at most by proof of Lemma \ref{le;commutator;modfied paraproduct, L}. To establish the $C^\alpha$ Hölder regularity of $u'$, it is necessary for the remainder $u^\sharp$ to possess Sobolev regularity of order at least $H^{\alpha + \frac{1}{2} + }$, this is valid only when $\gamma>\frac{4}{3}$. 
	\end{remark}
	\begin{remark}
		We can extend the range to $\gamma\in(1,2]$, which means $\alpha\in(0,\frac{1}{2})$. We omit these proofs because in this paper, we primarily aim to demonstrate the effectiveness of the paracontrolled method in enhancing the regularity of solutions.
	\end{remark}
	\subsection{Organization of the paper}
	This paper is organized as follows. In Section \ref{sec;basic tools}, we introduce some notations and basic tools. In Section \ref{sec;regular case}, we prove Theorem \ref{th;local result} and Theorem \ref{th;global result} for the fractional stochastic Burgers equation, and we also prove the result in Remark \ref{rem;result;DP} for the rough Degasperis–Procesi equation. In Section \ref{sec;Para-controlled}, we prove Theorem \ref{th;paracontrolled} for the paracontrolled solution. In Section \ref{sec;conjecture}, we provide a conjecture on the range of $\gamma$ in Sobolev space and discuss some future research directions.
	\section{Basic tool}\label{sec;basic tools}
	In this section, we introduce some basic stochastic and deterministic tools, more useful technical estimation tools will be introduced in the appendix.
	\subsection{Littlewood-Paley theory}\label{subsec;Littlewood-Paley }
	\par The first key tool is Littlewood-Paley decomposition. We introduce an operator named \textbf{``nonhomogeneous dyadic blocks''} $\Delta_j$ for $j\geq -1$, and satisfying that for any $u$ belong to $\ml S'(\mathbb{T}^d)$ (It's also valid for $\m R^d$), we have decomposition 
	$$u=\sum_{j\geq -1} \Delta_j u.$$
	More precisely, it's associated with two functions $\rho_l(k)$ and $\chi(k)$ under Fourier mode such that 
	$$\ml F(\Delta_lu)(k)=\rho(2^{-l}k)u_k,\quad l\geq 0\quad and\quad \ml F(\Delta_{-1}u)(k)=\ml \chi(k)u_k.$$
	Here $\chi$ and $\rho$ are two radial functions, valued in the interval $[0,1]$, belonging respectively to $\ml D(B(0,\frac{4}{3}))$ and $\ml D(\{k\in \mathbb{Z}^d:\frac{3}{4}\leq k\leq \frac{8}{3}\})$, such that 
	$$\chi(k)+\sum_{l\geq 0}\rho(2^{-l}k)=1,\quad k\in \mathbb{Z}^d.$$
	It can define nonhomogeneous Besov space $B_{p,r}^s(\mathbb{T}^d)$ which contains H\"older space $\ml C^s(\mathbb{T}^d):=B_{\infty,\infty}^s(\mathbb{T}^d)$ for $s\in(0,1)$ and nonhomogeneous Sobolev space $H^s(\mathbb{T}^d):=B^s_{2,2}(\mathbb{T}^d)$ such that 
	$$B_{p,r}^s(\mathbb{T}^d):=\left\{ u\in S':\left(\sum_{l\geq-1}2^{lsr}\Vert \Delta_lu\Vert_{L^p(\mathbb{T}^d)}^r\right)^\frac{1}{r}<\infty\right\},$$
	also define the norm. Also define the Fourier multiplier
	\[
	\mathcal{F}(\dot \Delta_j u)(k) = \rho(2^{-l} k) u_k, \quad l \in \mathbb{Z},
	\]
	where \(\rho\) is a smooth cut-off function supported in an annulus. Then we introduce the homogeneous Besov space
	\[
	\dot B_{p,r}^s(\mathbb{T}^d) := \left\{ u \in \mathcal{S}_h' : \left( \sum_{l\in\m Z} 2^{lsr} \| \dot \Delta_l u \|_{L^p(\mathbb{T}^d)}^r \right)^{1/r} < \infty \right\},
	\]
	where \(\mathcal{S}_h'\) denotes the space of tempered distributions \(u\) such that for any \(\theta \in\ml D(\mathbb{T}^d)\) with value $1$ near 0,
	\begin{equation}\label{eq;theta}
		\lim_{\lambda \to \infty} \| \theta(\lambda D) u \|_{L^\infty} = 0.
	\end{equation}
	It can be verified that \(\dot B_{2,2}^s(\mathbb{T}^d)\) coincides with the homogeneous Sobolev space
	\[
	\dot H^s(\mathbb{T}^d) = \left\{ u \in \mathcal{S}_h' : \left( \sum_{k \neq 0} |k|^{2s} |\mathcal{F}(u)(k)|^2 \right)^{1/2} < \infty \right\}.
	\]
	It is straightforward to verify that condition \eqref{eq;theta} implies \(\int_{\mathbb{T}} u(x) \, dx = \widehat{u}(0) = 0\). Consequently, a function \(f\) belongs to \(\dot H^s\) if and only if \(f \in H^s\) and \(\int_{\mathbb{T}} f(x) \, dx = 0\).
	Under this setting, we introduce para-product decomposition which give a product rules in $\ml S'(\mathbb{T}^d)$. Let $S_j u=\sum_{k\leq j-1}\Delta_ku $, then we have the estimate
	\begin{lemma}[\cite{bahouri2011fourier}]\label{le;est;paraproduct}
		Define $f\prec g:=\sum_{j}S_{j-1}f\Delta_jg$, $f\succ g:=g\prec f$, $f\circ g:=\sum_{\vert k-j\vert\leq 1}\Delta_k f\Delta_j g$, let $p,r\in[1,\infty]$ and $s\in\m R$, then we have the following estimates
		\begin{itemize}
			\item{1. If $k\in \mathbb{N}$, then
				$$\Vert f\prec g\Vert_{B^s_{p,r}}\leq C\Vert f\Vert_{L^\infty}\Vert D^k g\Vert_{B_{p,r}^{s-k}}.$$}
			\item{2. If $k\in \mathbb{N}$, $t<0$ and $\frac{1}{r}=\min(1,\frac{1}{r_1}+\frac{1}{r_2})$, then $$\Vert f\prec g\Vert_{B_{p,r}^{s+t}}\leq C\Vert f\Vert_{B_{\infty,r_1}^t}\Vert D^k g\Vert_{B_{p,r_2}^{s-k}}.$$}
			\item{3. If $\frac{1}{p}=\frac{1}{p_1}+\frac{1}{p_2}\leq1$, $\frac{1}{r}=\frac{1}{r_1}+\frac{1}{r_2}\leq1$ and $s=s_1+s_2>0$, then
				$$\Vert f\circ g\Vert_{B_{p,r}^{s}}\leq C\Vert f\Vert_{B_{p_1,r_1}^{s_1}}\Vert g\Vert_{B_{p_2,r_2}^{s_2}}.$$}
		\end{itemize}
		The above estimate also holds if we replace \(\dot\Delta_j\) with \(\Delta_j\) and \(\dot S_j\) with \(S_j\).
	\end{lemma}
	By above estimate in $H^s$ space, we can check that Moser-type estimate hold.
	\begin{lemma}[Moser-type estimate]\label{le;Moser}
		Let $s>0$, $\epsilon>0$, then the following estimates hold:
		\begin{itemize}
			\item{(1) for $0<s<\frac{1}{2}$, $$\Vert fg\Vert_{H^{2s-\frac{1}{2}}}\leq C\Vert f\Vert_{H^s}\Vert g\Vert_{H^s}.$$}
			\item{(2) for $s>\frac{1}{2}$, $$\Vert fg\Vert_{H^{s}}\leq C\Vert f\Vert_{H^s}\Vert g\Vert_{H^s}.$$} 
			\item{(3) for $s=\frac{1}{2}$, $$\Vert fg\Vert_{H^{s-\epsilon}}\leq C\Vert f\Vert_{H^s}\Vert g\Vert_{H^s}.$$}
		\end{itemize}
	\end{lemma}
	\par To explain the paraproduct, we first note that two distributions $f$ and $g$ cannot, in general, be multiplied directly. A sufficient condition for such a product to be well defined is that the requirement of  Lemma \ref{le;est;paraproduct} be satisfied, which at least necessitates $s_1+s_2\geq0$ . Moreover, $f\prec g$ and $f\succ g$ are called para-products since they are always well defined regardless of the regularity of $f$ and $g$. The term $f\circ g$ is referred to as the resonant term, because it involves interactions of frequencies that are close to each other, which typically restricts the range of regularities for which it can be controlled.
	\par For our convenience, we give the following notation
	$$W_p^\alpha=B_{p,p}^\alpha,\quad W^\alpha=W_2^\alpha\cap W_\infty^\alpha,\quad$$
	and note that $W^\alpha$ is a Banach algebra when $\alpha>0$.
	\begin{lemma}
		If $f\in W^\alpha$, then $f\in W_p^\alpha$ for any $p\geq2$. Moreover, we have estimate
		\begin{align*}
			\Vert f\Vert_{W_p^\alpha}\leq\Vert f\Vert_{H^{\alpha}}^{\frac{2}{p}}\Vert f\Vert_{\ml C^{\alpha}}^{\frac{p-2}{p}}.
		\end{align*}
		\begin{proof}
			By definition of $B_{p,p}^\alpha$ and interpolation in $L^p$, we have
			\begin{align*}
				\Vert f\Vert_{B_{p,p}^\alpha}^p=&\sum_{j\geq-1}2^{j\alpha p}\Vert \Delta_jf\Vert^{p}_{L^p}\\
				\leq& \sum_{j\geq-1}2^{j\alpha p}(\Vert \Delta_jf\Vert^{\frac{2}{p}}_{L^2}\Vert \Delta_jf\Vert^{\frac{p-2}{p}}_{L^\infty})^p\\
				\leq& (\sup_{j\geq-1}2^{j\alpha}\Vert \Delta_jf\Vert_{L^\infty})^{p-2}\sum_{j\geq-1}2^{j\alpha 2}\Vert \Delta_jf\Vert^{2}_{L^2}\\
				\leq&\Vert f\Vert_{H^{\alpha}}^{2}\Vert f\Vert_{B_{\infty,\infty}^{\alpha}}^{p-2}.
			\end{align*}
			Then our lemma is proved.
		\end{proof}
	\end{lemma}
	\par Sometimes it is also necessary to use a modified version of the paraproduct. Such operators have been introduced in \cite{gubinelli2015paracontrolled,gubinelli2017kpz}, and are defined as follows:
	Let $\varphi\in C_c^\infty(\m R,\m R_+)$ be nonnegative with compact support contained in $\m R_+$ and with total mass 1, define $Q_i$ for all $i\geq1$ as follows
	\begin{align*}
		Q_i:C\ml C^\beta\to C\ml C^\beta,\quad Q_if(t)=\int_{\m R}2^{\gamma i}\varphi(2^{\gamma i}(t-s))f(s\vee0)\mathrm{d}s.
	\end{align*}
	We will often apply $Q_i$ and other operators on $C\ml C^\beta$ to functions $f\in C_T\ml C^\beta$ which we then simply extend from $[0, T]$ to R+ by considering $f (\cdot\wedge T )$. With the help of $ Q_i$, we define a modified para-product
	\begin{equation}\label{eq;para-product;modify}
		f\pprec g=\sum_{i}(Q_iS_{i-1}f)\Delta_ig.
	\end{equation}
	Similar to estimate for the paraproduct in Lemma \ref{le;est;paraproduct}, we also have estimates for this operator.
	\begin{lemma}\label{le;est;paraproduct;modify}
		For any $\beta\in \m R$ and $p\geq 2$, we have 
		\begin{align*}
			\Vert f\pprec g(t)\Vert_{W^\beta_p}\leq C\Vert f\Vert_{C_tL^\infty}\Vert  g\Vert_{W^\beta_p},
		\end{align*}
		for all $t>0$. And for $-1<\alpha<0$ furthermore
		\begin{align*}
			\Vert f\pprec g(t)\Vert_{W^{\alpha+\beta}_p}\leq C\Vert f\Vert_{C_t\ml C^\alpha}\Vert  g\Vert_{W^\beta_p}.
		\end{align*}
		\begin{proof}
			Noting that $\ml F(\ml Q_jS_{j-1}f(s)\Delta_jg(t))\in  2^j\tilde{\ml C}$, where $\tilde{\ml C}$ is some annuals, we are left with proving an appropriate estimate for $\Vert \ml Q_jS_{j-1}f(s)\Delta_jg(t)\Vert_{L^p}$. By H\"older inequality, and definition of $\varphi$, we have 
			\begin{align*}
				\Vert \int_0^t 2^{\gamma j}\varphi(2^{\gamma j}(t-s))S_{j-1}f(s)\mathrm{d}s \Delta_j g(t)\Vert_{L^p}\lesssim& \int_0^t 2^{\gamma j}\varphi(2^{\gamma j}(t-s))\Vert S_{j-1}f(s)\Vert_{L^\infty}\mathrm{d}s \Vert \Delta_j g(t)\Vert_{L^p}\\
				\lesssim&\int_{\m R} 2^{\gamma j}\varphi(2^{\gamma j}(t-s))\mathrm{d}s \Vert f\Vert_{C_tL^\infty}\Vert \Delta_j g(t)\Vert_{L^p}\\
				\lesssim&\Vert f\Vert_{C_tL^\infty}\Vert \Delta_j g(t)\Vert_{L^p}.
			\end{align*}
			Since $g(t)\in W_p^\beta$, by Lemma 2.23 in \cite{bahouri2011fourier}, we get the first estimate. The second estimate is similar, requiring only the observation that 
			\begin{equation}
				\sup_{j\geq -1}2^{j\alpha}\Vert S_{j-1}f(s)\Vert_{L^\infty}\leq C\Vert f(s)\Vert_{\ml C^\alpha},
			\end{equation}
			for $-1<\alpha<0$ by Proposition 2.79 in \cite{bahouri2011fourier}.
		\end{proof}
	\end{lemma}
	\par The remaining estimates will be presented in the Appendix \ref{sec;commutator est}. Finally, we introduce a class of temporal regularity spaces, which have been referenced in \cite{gubinelli2015paracontrolled,gubinelli2017kpz,hairer2013solving,catellier2018paracontrolled}. Define the such norm with time $T$ as follows. $C_TX=C([0,T],X)$ for the space of continuous maps from $[0,T]$ to $X$, equipped with the norm $\sup_{t\in[0,T]}\Vert \cdot\Vert_X$. For $\alpha\in(0,1)$, we also define $C_T^\alpha X$ as a space of $\alpha$-H\"older continuous functions from $[0,T]$ to $X$ equipped with the norm $\Vert f\Vert_{C_T^\alpha X}=\sup_{0\leq s<t\leq T}\frac{\Vert f(t)-f(s)\Vert_X}{\vert t-s\vert^\alpha}$. For our convenience, for $\alpha>0$ we have the following notation
	$$ \ml W_p^\alpha(T)=C_TB_{p,p}^\alpha\cap C^{\frac{\alpha}{\gamma}}_TL^p,\quad \ml W^\alpha(T)= \ml W_2^\alpha(T)\cap \ml W_\infty^\alpha(T).$$
	For $\alpha=0$, $\ml W_p^0(T)=C_TL^p$ and for $\alpha<0$, we also give a definition $\ml W^{\alpha}_p(T)=C_TB_{p,p}^\alpha$. It's easy to check that $\ml W^\alpha(T)$ is a Banach algebra when $\alpha\geq 0$. Next, we introduce some useful estimate
	\begin{lemma}\label{le;est;paraproduct;modifed}
		Let $\alpha\in\m R$, $\epsilon,T>0$, $\rho\in(0,\gamma)$ satisfy  $\alpha+\rho<\gamma$. Let $f\in \ml W_\infty^\epsilon(T)$, $g(0)=0$, $g\in C_TW^{\alpha+\gamma-\epsilon}(T)$ and $\ml Lg\in C_T W^\alpha$. Then
		\begin{align*}
			\Vert f\pprec g\Vert_{\ml W_\infty^{\alpha+\rho}(T)}\lesssim T^{1-\frac{\rho}{\gamma}}\Vert f\Vert_{\ml W_\infty^\epsilon(T)}(\Vert g\Vert_{C_T\ml C^{\alpha+\gamma-\epsilon}}+\Vert \ml Lg\Vert_{C_T\ml C^{\alpha}}),
		\end{align*} 
		and
		\begin{align*}
			\Vert f\pprec g\Vert_{\ml W_2^{\alpha+\rho}(T)}\lesssim T^{1-\frac{\rho}{\gamma}}\Vert f\Vert_{\ml W_\infty^\epsilon(T)}(\Vert g\Vert_{C_TH^{\alpha+\gamma-\epsilon}}+\Vert \ml Lg\Vert_{C_TH^{\alpha}}).
		\end{align*} 
		\begin{proof}
			We only prove the case of $\ml W_2^{\alpha+\rho}(T)$ and the $p=\infty$ is similar. By Lemma \ref{le;heat flow}, Lemma \ref{le;est;paraproduct;modify} and Lemma \ref{le;commutator;modfied paraproduct, L}, we have
			\begin{align*}
				\Vert f\pprec g\Vert_{\ml W_2^{\alpha+\rho}(T)}\lesssim& \Vert f\pprec g(0)\Vert_{\ml W_2^{\alpha+\rho}(T)}+T^{1-\frac{\rho}{\gamma}}\Vert \ml L(f\pprec g)\Vert_{C_TH^{\alpha}}\\
				\leq&T^{1-\frac{\rho}{\gamma}} \Vert \ml L(f\pprec g)-f\pprec \ml Lg\Vert_{C_TH^{\alpha}}+\Vert f\pprec \ml Lg\Vert_{C_TH^{\alpha}}\\
				\leq &T^{1-\frac{\rho}{\gamma}} \Vert f\Vert_{C_T\ml C^{\epsilon}}\Vert g\Vert_{C_TH^{\alpha+\gamma-\epsilon}}+\Vert f\Vert_{C_T\ml C^\epsilon}\Vert \ml Lg\Vert_{C_TH^{\alpha}}\\
				\leq &T^{1-\frac{\rho}{\gamma}}\Vert f\Vert_{C_T\ml C^{\epsilon}}(\Vert g\Vert_{C_TH^{\alpha+\gamma-\epsilon}}+\Vert \ml Lg\Vert_{C_TH^{\alpha}}),
			\end{align*}
			then we prove the lemma.
		\end{proof}
	\end{lemma}
	\subsection{Stochastic tools}
	In this little section, we introduce some stochastic tools we need in this paper. The first tool is the Kolmogorov's continuity criterion.
	\begin{lemma}[\cite{le2016brownian}]\label{le;Kolmogorov}
		Let $X=(X(t))_{t\in[0,T]}$ be a random process and take values in a complete metric space $(E,d)$. If there exist $p,\epsilon,C>0$ such that, for every $s,t\in[0,T]$, 
		\begin{equation}
			\m E[d(X(s),X(t))^p]\leq C\vert t-s\vert^{1+\epsilon}.
		\end{equation}
		Then, there is a modification $\tilde X$ of $X$ whose sample paths are H\"older continuous with exponent $\alpha$ for every $\alpha\in (0,\frac{\epsilon}{p})$. This means that, for every $\omega\in\Omega$ and every $\alpha\in(0,\frac{\epsilon}{p})$, there exists a finite constant $C_\alpha(\omega)$ such that, for every $s,t\in I$, 
		\begin{align*}
			d(\tilde X(s,\omega),\tilde X(t,\omega))\leq C_\alpha(\omega)\vert t-s\vert^\alpha,
		\end{align*}
		where $\tilde X(t)$ is a modification of $X(t)$ with continuous sample paths.
	\end{lemma}
	Now we give a basic analysis for white noise $\xi$. For the linear evolution $X(t):=\int_0^t e^{(t-s)\Lambda^\gamma}\xi(s)\mathrm{d}s$, we have the following lemma.
	\begin{lemma}\label{le;regularity}
		The spatial Fourier transform $X_k=\ml F_xX(k)$ is a Gaussian process with zero-mean and satisfy
		\begin{align*}
			\m E[X_k(t)X_{k'}(t')]=&(e^{-\vert t-t'\vert\vert k\vert^\gamma}-e^{-( t+t')\vert k\vert^\gamma})\frac{\vert \mathbb{T}\vert}{2\vert k\vert^\gamma}\textbf{I}_{k+k'=0}(k,k')\\
			 \m E[X_0(t)X_{k}(t')]=& (t\wedge t')\vert \m T\vert\textbf{I}_{k=0}(k).
		\end{align*}
		where $\vert \mathbb{T}\vert=\int_{\mathbb{T}}1\mathrm{d}x$ and $\textbf{I}_{k+k'=0}(k,k')$ is the indicator function. Moreover, for any $ T>0$ and $\beta<\frac{\gamma(1-2\kappa)}{2}-\frac{1}{2}$, we have $X(t)\in C_T^\kappa W^\beta$ for $\kappa<\frac{1}{2}$.
		\begin{proof}
			It's easy to check the first result by noting that
			\begin{equation}\label{eq;covariance;xi}
				\m E[\ml F\xi(t,k){\ml F\xi(s,k')}]=\delta(t-s)\textbf{I}_{k+k'=0}(k,k').
			\end{equation}
			Considering the second result, we can rewrite $X(t)$ as 
			\begin{align*}
				X(t)=\sum_{k\in \mathbb{Z}}X_k(t)e^{ikx}.
			\end{align*}
			For $j\geq0(k\neq0)$, we calculate $\m E\vert \Delta_j (X(t)-X(t'))\vert^2$ as follows
			\begin{align*}
				\m E\vert \Delta_j (X(t)-X(t'))\vert^2=&\m E\sum_{k,k'\in \m Z}\rho_j(k)\rho_j(k')[X_k(t)-X_k(t')][X_{-k'}(t)-X_{-k'}(t')]e^{i(k-k')x}\\
				\lesssim& 2^{j(1-\gamma(1-\delta))}\vert t-t'\vert^\delta,
			\end{align*}
			for some $\delta\in[0,1]$. The case $j=-1$ can be treated using essentially the same arguments, except that then we need to distinguish the cases $k = 0$ and $k \neq 0$. By Gaussian hypercontractivity \cite{janson1997gaussian}, we have
			\begin{align*}
				\m E\vert \Delta_j (X(t,x_i)-X(t',x_i))\vert^{2p}\lesssim(\m E\vert \Delta_j (X(t,x_i)-X(t',x_i))\vert^{2})^p.
			\end{align*}
			It implies that
			\begin{align*}
				\m E\Vert \Delta_j (X(t)-X(t'))\Vert_{L^{2}}^{2p}=&\m E(\int_{\m T}\vert \Delta_j (X(t)-X(t'))\vert^{2}\mathrm{d}x)^p\\
				\lesssim&\int_{\m T^{\otimes p}}\m E\prod_{i=1}^{p}\vert \Delta_j (X(t,x_i)-X(t',x_i))\vert^{2}\mathrm{d}x_i\\
				\lesssim &\int_{\m T^{\otimes p}}\prod_{i=1}^{p}(\m E\vert \Delta_j (X(t,x_i)-X(t',x_i))\vert^{2p})^\frac{1}{p}\mathrm{d}x_i\\
				\lesssim &\int_{\m T^{\otimes p}}\prod_{i=1}^{p}(\m E\vert \Delta_j (X(t,x_i)-X(t',x_i))\vert^{2})\mathrm{d}x_i\\
				\lesssim & 2^{pj(1-\gamma(1-\delta))}\vert t-t'\vert^{\delta p}.
			\end{align*}
			Then we can estimate $\m E\Vert X(t)-X(t')\Vert_{H^\beta}^{2p}$ as 
			\begin{align*}
				\m E\Vert X(t)-X(t')\Vert_{B_{2,2}^\beta}^{2p}\lesssim&(\sum_{j\geq -1}2^{2j\beta }\m (E\Vert \Delta_j (X(t)-X(t'))\Vert_{L^{2}}^{2p})^{\frac{1}{p}})^p\\
				\leq&(\sum_{j\geq -1}2^{j(2\beta+1-\gamma(1-\delta))})^p\vert t-t'\vert^{\delta p}\leq C\vert t-s\vert^{\delta p},
			\end{align*}
			for $\beta<\frac{\gamma(1-\delta)-1}{2}$. By Lemma \ref{le;Kolmogorov}, we have $X(t)\in C^{\frac{\delta}{2}-\frac{1}{2p}}_TH^{\frac{\gamma(1-\delta)-1}{2}}$. Choose $p$ large enough, we can prove the case of $H^s$. For H\"older space, we can also have the results by similar argument
			\begin{align*}
				\m E\Vert X(t)-X(t')\Vert_{B_{2p,2p}^\beta}^{2p}=&\sum_{j\geq -1}2^{2j\beta p}\m E\Vert \Delta_j (X(t)-X(t'))\Vert_{L^{2p}}^{2p}\\
				\lesssim&\sum_{j\geq -1}2^{2j\beta p}\Vert \m E\vert \Delta_j(X(t)-X(t'))\vert^2\Vert_{L^p}^p\\
				\lesssim&\sum_{j\geq -1}2^{jp(2\beta+1-\gamma(1-\delta))}\vert t-t'\vert^{\delta p}\leq C\vert t-s\vert^{\delta p}.
			\end{align*}
			Let $p$ be large enough and by embedding for Besov space, we get $X(t)\in C^{\frac{\delta}{2}-\frac{1}{2p}}_T\ml C^{\frac{\gamma(1-\delta)-1}{2}}$. Choosing $p$ large enough again, we can prove the case of $\ml C^s$, which we finish our proof.
		\end{proof}
	\end{lemma}
	\begin{remark}\label{rem;regularity}
		By Lemma \ref{le;regularity}, we have $X(t)\in C_T^{\kappa}W^{\beta}$ for $\frac{\beta}{\gamma}<\frac{1}{2}-\frac{1}{\gamma}-\kappa$. Let $\beta=0$, then $\kappa$ can choose to be $\frac{1}{2}-\frac{1}{2\gamma}-\epsilon$ for any small $\epsilon>0$. Then we have $X(t)\in \ml W^\alpha(T)$ a.s for any $T\geq0$ and $\alpha<\frac{\gamma}{2}-\frac{1}{2}$ if we fix $\epsilon$ small enough. 
	\end{remark}
	\subsection{Iteration argument}
	Next we introduce a generalized contraction mapping theorem which compared with Lemma 5.5($ L=0$) in \cite{bahouri2011fourier}.
	\begin{lemma}\label{le;Picard th}
		Let $E$ be a Banach space, and $\ml B$ a continuous bilinear map from $E\times E$ to $E$, and $r$ a positive real number such that
		$$r\leq\frac{1}{8\Vert \ml B\Vert}\quad with \quad\Vert \ml B\Vert_E=\sup_{\Vert u\Vert_E,\Vert v\Vert_E\leq1}\Vert \ml B(u,v)\Vert_E.$$ 
		Let $ L$ is a linear map $E$ to $E$ and satisfied 
		$$\sup_{v\in E}\frac{\Vert  L(v)\Vert_E}{\Vert v\Vert_E}\leq \frac{1}{4}.$$
		Then for any $a$ in a ball $B(0,r)\subset E$, there exists a unique $\tilde x$ in $\ml B(0,2r)$ such that 
		$$\tilde x=a+ L(\tilde x)+\ml B(\tilde x,\tilde x).$$
		\begin{proof}
			The proof of first result can refer to \cite{bahouri2011fourier}. For second, we involved an application of iterative scheme defined by 
			$$\tilde x_0=a\quad \tilde x_{n+1}=a+ L(\tilde x_n)+\ml B(\tilde x_n,\tilde x_n).$$
			By induction argument, we show that $\Vert\tilde  x_{n+1}\Vert\leq 2r$ when $\Vert \tilde x_n\Vert\leq 2r$. In fact, by our assumption, we get
			\begin{align*}
				\Vert\tilde x_{n+1}\Vert_E\leq r+\frac{r}{4}+4\Vert \ml B\Vert r^2\leq 2r.
			\end{align*}
			Thus, $(\tilde x_n)_{n\in \m N}$ remains in $B(0,2r)$. Moreover, 
			\begin{align*}
				\Vert \tilde x_{n+1}-\tilde x_n\Vert_E\leq& \Vert L(\tilde x_n-\tilde x_{n-1})\Vert_E +4r\Vert \ml B\Vert\Vert \tilde x_n-\tilde x_{n-1}\Vert_E\\
				\leq&\frac{3}{4}\Vert \tilde x_n-\tilde x_{n-1}\Vert_E.
			\end{align*}
			Then we get the limit of $(\tilde x_n)$ is a unique fixed point of $\tilde x\mapsto a +L(\tilde x)+\ml B(\tilde x,\tilde x)$. This means that the lemma is proved.
		\end{proof}
	\end{lemma}
	We also provide version for studying specific types of coupled equations. Let $f((x,y))$ be a linear functional on $E\times E$ which means
	\begin{equation}\label{eq;linear;f}
		\begin{aligned}
			&f(c(x,y))=cf((x,y)),\\
			&f(x_1+x_2,y_1+y_2)=f(x_1,y_1)+f(x_2,y_2).
		\end{aligned}
	\end{equation} 
	Assume that $f(x,y)$ satisfies some control property from $\dot E\times E$ to $Y$, such as 
	\begin{equation}\label{eq;controlled;paracontrolled ansatz}
		\Vert f(x,y)\Vert_Y\leq \Vert x\Vert_{\dot E}+c \Vert y\Vert_E,
	\end{equation} we have the following lemma
	\begin{lemma}\label{le;Picard th;specific}
		Let $\dot E, E,Y$ be Banach spaces satisfying $ \dot E\subset Y \subset E$, $
		\ml E:=\dot E\times E$ be a Banach space with norm 
		$$\Vert (x,y)\Vert_{\ml E}=2\Vert x\Vert_{\dot E}+\Vert y\Vert_E.$$
		Moreover let $\ml B$ be a continuous bilinear map from $Y\times Y$ to $\dot E$, and $r$ a positive real number such that
		$$\Vert \ml B(u,v)\Vert_{\dot  E}\leq \frac{1}{16r}\Vert u\Vert_Y \Vert v\Vert_Y,$$
		Let $\ml G:Y\to\dot E$, $g:E\to \dot E$ and $h:\dot E\to\dot E$ be linear maps  satisfying 
		$$\sup_{v\in E}(\frac{\Vert \ml G(v)\Vert_{\dot E}}{\Vert v\Vert_Y},\frac{\Vert  g(v)\Vert_{\dot E}}{\Vert v\Vert_E})\leq \frac{1}{16},\quad \frac{\Vert h(v)\Vert_{\dot E}}{\Vert v\Vert_{\dot E}}\leq \frac{1}{16}.$$
		Let $f$ satisfy \eqref{eq;linear;f} and \eqref{eq;controlled;paracontrolled ansatz} for $Y$, then there exist $\tilde c>0$, such that for any $(a,b)\in B(0,\frac{r}{4})\subset \ml E$ and $0<c<\tilde c$, there exists unique $(x,y)\in B(0,4r)\subset \ml E$ such that 
		\begin{equation}\label{eq;key mapping}
			y=b+f(x,y),\quad  x=a+g(y)+h(x)+\ml B(f(x,y),f(x,y))+\ml G(f(x,y)).
		\end{equation}
		\begin{proof}
			We involved application of iterative scheme defined by
			\begin{align*}
				x_0=&a,\quad y_0=b,\quad y_{n+1}=b+f(x_n,y_n),\\ x_{n+1}=&a+g(y_n)+h(x_n)+\ml G(f(x_n,y_n))+\ml B(f(x_n,y_n),f(x_n,y_n)).
			\end{align*}
			We first prove that $(x_{n+1},y_{n+1})\in B(0,4r)\subset \ml E$ when $(x_{n},y_{n})\in B(0,4r)\subset \ml E$. Observing that $\Vert x_n\Vert\leq 2r$, $\Vert y_n\Vert\leq 4r$ and 
			\begin{align*}
				\Vert f(x_{n},y_{n})\Vert_E\leq \Vert f(x_{n},y_{n})\Vert_Y\leq \Vert x_n\Vert_{\dot E}+c\Vert y_n\Vert_E\leq 2r+4cr,
			\end{align*}
			by a directly computation, we have 
			\begin{align*}
				\Vert (x_{n+1},y_{n+1})\Vert_{\ml E}\leq& \Vert b\Vert_{E}+\Vert  f(x_n,y_n)\Vert_{E}+2\Vert a\Vert_{\dot E}+2\Vert g(y_n)\Vert_{\dot E}+2\Vert h(x_n)\Vert_{\dot E}\\
				&+2\Vert \ml G(f(x_n,y_n))\Vert_{\dot E}+2\Vert \ml B(f(x_n,y_n),f(x_n,y_n))\Vert_{\dot E}\\
				\leq &2\Vert a\Vert_{\dot E}+\Vert b\Vert_{E}+\frac{1}{8}\Vert y_n\Vert_{E}+\frac{1}{8}\Vert x_n\Vert_{\dot E}+\frac{9}{8}\Vert f(x_n,y_n)\Vert_{Y}+\frac{1}{8r}(\Vert f(x_n,y_n)\Vert_Y^2)\\
				\leq& \frac{r}{4}+\frac{r}{2}+\frac{r}{4}+\frac{9}{8}(2r+4cr)+\frac{1}{8r}(4r^2+16cr^2+16c^2r^2)\\
				\leq&\frac{15}{4}r+(\frac{13}{2}c+2c^2)r<4r,
			\end{align*}
			for $0<c<\tilde c=\frac{1}{30}$. Noting that $(x_0,y_0)\in B(0,\frac{r}{4})\subset B(0,4r)$, we obtain that $(x_n,y_n)\in B(0,4r)$ for all $n\in \m N$. Setting  $f_{n,m}=f(x_{n},y_{n})-f(x_{m},y_{m})\in E$, for all $n,m\in\m N$, we have
			\begin{align*}
				\Vert f_{n,m}\Vert_Y\leq \Vert x_n-x_m\Vert_{\dot E}+c \Vert y_n-y_m\Vert_{E},\quad \Vert f(x_n,y_n)\Vert_Y\leq \Vert x_n\Vert_{\dot E}+c\Vert y_n\Vert_E\leq 2r+4cr
			\end{align*}
			To prove that $\{(x_n,y_n)\}_{n\in \m N}$ is a Cauchy sequence, we estimate that
			\begin{align*}
				\Vert (x_{n+1},y_{n+1})-(x_{n},y_{n})\Vert_{\ml E}\leq&2\Vert x_{n+1}-x_n\Vert_{\dot E}+\Vert y_{n+1}-y_n\Vert_{E}\\
				\leq &2\Vert g(y_n-y_{n-1})\Vert_{\dot E}+2\Vert h(x_n-x_{n-1})\Vert_{\dot E} +2\Vert \ml G(f_{n,n-1})\Vert_{\dot E}\\
				&+2\Vert \ml B(f_{n,n-1},f(x_{n},y_{n}))+\ml B(f(x_{n-1},y_{n-1}),f_{n,n-1})\Vert_{\dot E}+ \Vert f_{n,n-1}\Vert_E\\
				\leq & \frac{1}{8}\Vert y_n-y_{n-1}\Vert_{E}+\frac{1}{8}\Vert x_n-x_{n-1}\Vert_{\dot E}+\frac{1}{8r}(\Vert f_{n,n-1}\Vert_Y\Vert f(x_n,y_n)\Vert_Y)\\
				&+\frac{1}{8r}(\Vert f_{n,n-1}\Vert_Y\Vert f(x_{n-1},y_{n-1})\Vert_Y)+\frac{9}{8}\Vert f_{n,n-1}\Vert_Y\\
				<&(\frac{7}{8}+\frac{13c}{8})\Vert (x_{n},y_{n})-(x_{n-1},y_{n-1})\Vert_{\ml E}
				<\frac{15}{16}\Vert (x_{n},y_{n})-(x_{n-1},y_{n-1})\Vert_{\ml E}.
			\end{align*}
			Then we get the limit of $(x_{n},y_{n})$ is a unique fixed point of mapping \eqref{eq;key mapping} in Banach space $\ml E$, which proves the lemma.
		\end{proof}
	\end{lemma}
	\section{Result in the range of $\gamma\in(\frac{4}{3},2]$}\label{sec;regular case}
	This section treats the highly dissipative regime, in which the linear component
	$X(t)$ enjoys improved regularity. Specifically, for $\gamma\in(\frac{4}{3},2]$, Lemma \ref{le;regularity} yields $X\in C_TH^{\frac{1}{6}+\delta}$ for a sufficiently small $\delta$. This regime is subcritical because the critical regularity for $X$ is
	$H^{\frac{3}{2}-\gamma}.$
	\subsection{Bi-linear estimate}
	First we build a standard bilinearity estimate. Defining $B(f,g)$ as follows
	\begin{equation}\label{eq;bilinear}
		B(f,g)(t)=\int_0^te^{(t-s)\Lambda^\gamma}\p_x(f(s)g(s))\mathrm{d}s,
	\end{equation}
	then we have the following estimate.
	\begin{lemma}\label{le;est;bilinear;general}
		Fixing $E_T:=L^\infty_TH^s$, with $0\leq \rho<\gamma-1$, $s=\sigma+\rho$, then we have
		\begin{equation}
			\Vert B(f,g)\Vert_{E_T}\leq C T^{\frac{\gamma-1-\rho}{\gamma}}(1+T^\frac{\rho}{\gamma})\Vert fg\Vert_{C_TH^\sigma}.
		\end{equation}
		\begin{proof}
			By \eqref{eq;bilinear} and Parseval formula, we have 
			\begin{align*}
				\Vert B(f,g)(t)\Vert_{H^s}\lesssim&\int_0^t\Vert e^{-(t-\tau)\Lambda^\gamma}(1-\p_x^2)^\frac{\rho}{2}\p_x(1-\p_x^2)^\frac{\sigma}{2}(fg)\Vert_{L^2}\mathrm{d}\tau\\
				\lesssim& \int_0^t\Vert e^{-(t-\tau)\vert k\vert^\gamma}(1+\vert k\vert^2)^\frac{\rho}{2}\p_x(1+\vert k\vert^2)^\frac{\sigma}{2}(fg)\Vert_{L^2}\mathrm{d}\tau
			\end{align*}
			Divide $k$ in two parts: $\vert k\vert\lesssim1$ and $\vert k\vert\gtrsim 1$. For low frequencies $\vert k\vert\lesssim 1$, we have
			$$\vert e^{-(t-\tau)\vert k\vert^\gamma}(1+k^2)^\frac{\rho}{2}k\vert\lesssim \vert e^{-(t-\tau)\vert k\vert^\gamma}\vert k\vert\vert.$$ 
			For high frequencies $\vert k\vert\gtrsim 1$, we have
			$$\vert e^{-(t-\tau)\vert k\vert^\gamma}(1+k^2)^\frac{\rho}{2}k\vert\lesssim \vert e^{-(t-\tau)\vert k\vert^\gamma}\vert k\vert^{\rho+1}\vert.$$
			Then for $t\in[0,T]$ we obtain that
			\begin{align*}
				\Vert B(f,g)(t)\Vert_{H^s}\lesssim\int_0^t\max(\frac{1}{(t-\tau)^\frac{1}{\gamma}},\frac{1}{(t-\tau)^\frac{1+\rho}{\gamma}})\Vert fg\Vert_{H^\sigma}\mathrm{d}\tau\lesssim T^{\frac{\gamma-1-\rho}{\gamma}}(1+T^\frac{\rho}{\gamma})\Vert fg\Vert_{C_TH^\sigma},
			\end{align*}
			which finish our proof.
		\end{proof}
	\end{lemma}
	\begin{remark}\label{rmk;no endpoint}
		The endpoint case $\rho=\gamma-1$  requires a more refined analysis, typically involving frequencies localization in Besov spaces $B_{2,\infty}^s$.  Such endpoint issues are not addressed in the present work. Indeed, for any $\alpha<\frac{\gamma}{2}-\frac{1}{2}$, one has $X\in \ml W^{\alpha+}(T)$ almost surely, which suffices for our purposes.
	\end{remark}
	By Moser-type estimate, it's not difficult to prove that 
	\begin{lemma}\label{le;est;bilinear}
		Let $\gamma\in(\frac{4}{3},2]$, $\alpha>\frac{1}{6}$, $\delta$ satisfy $0<2\delta\leq\alpha+\gamma-\frac{3}{2}$, then for $T\leq 1$ we have the following bounds
		\begin{itemize}
			\item{(1) For $f,g\in L^\infty_T H_x^{\frac{1}{2}+\delta}$, we have
				\begin{align*}
					\Vert B(f,g)\Vert_{L^\infty_TH^{\frac{1}{2}+\delta}}\leq C_1T^{\frac{\gamma-1}{\gamma}}\Vert f\Vert_{L^\infty_T H^{\frac{1}{2}+\delta}}\Vert g\Vert_{L^\infty_T H^{\frac{1}{2}+\delta}}.
			\end{align*}}
			\item{(2) For $f,g\in L^\infty_T W^{\alpha}$, we have
				\begin{align*}
					\Vert B(f,g)\Vert_{L^\infty_TH^{\frac{1}{2}+\delta}}\leq C_2T^{\frac{\delta}{\gamma}}\Vert f\Vert_{C_TW^\alpha}\Vert g\Vert_{C_TW^\alpha}.
				\end{align*}
			}
			\item{(3) For $f\in L_T^\infty {H^{\frac{1}{2}+\delta}}$ and $g\in L_T^\infty W^\alpha$, then we have
				\begin{align*}
					\Vert B(f,g)\Vert_{L^\infty_TH^{\frac{1}{2}+\delta}}\leq C_3T^{\frac{\delta}{\gamma}} \Vert f\Vert_{L^\infty_T H^{\frac{1}{2}+\delta}}\Vert g\Vert_{L^\infty_T H^{\alpha}}.
			\end{align*}}
		\end{itemize}
	\end{lemma}
	\subsection{Well-posedness}
	In this section, we study the well-posedness of equation \eqref{eq;Burgers;white noise} for $\gamma>\frac{4}{3}$. We begin by considering the equation
	\begin{equation}\label{eq;C-M}
		\p_t v-\Lambda^\gamma v=\p_x(v^2)+2\p_x(v\tilde X)+\p_x(\tilde X^2),\quad u(0)=u_0.
	\end{equation}
	and solve it via a contraction mapping argument.
	\begin{proposition}\label{prop;local result;reg}
		Let $\alpha>\frac{1}{6}$, $\gamma\in (\frac{4}{3},2]$ and $u_0\in H^{\frac{1}{2}+\delta}$ for $2\delta\leq\alpha+\gamma-\frac{3}{2}$. If $\tilde X(t)\in{C_{\bar{T}}W^\alpha}$ for some $\bar{T}>0$ ,  then there exists a unique solution $v(t,x)$ of equation \eqref{eq;C-M} with a postive $T^*$ satisfying $v\in C_{T^*}H^{\frac{1}{2}+\delta}$.
		\begin{proof}
			Let $r=2\max(\Vert u_0\Vert_{H^{\frac{1}{2}+\delta}},\Vert \tilde X\Vert_{C_{\bar T}W^\alpha})$. We rewrite equation \eqref{eq;C-M} as 
			\begin{equation}
				v(t,x)=P(t)u_0+B(v,\tilde X)+B(\tilde X,v)+B(\tilde X,\tilde X)+B(v,v),
			\end{equation}
			where $P(t) f$ satisfying linear equation 
			\begin{equation}
				\p_t P(t)f-\Lambda^\gamma P(t)f=0,\quad P(0)f=f.
			\end{equation}
			Firstly, we claim that there exist $0<T^*\leq 1\wedge \bar T$ such that for any $t\in[0,T^*]$
			\begin{equation}\label{eq;3.1;1}
				a:=P(t)u_0+B(\tilde X,\tilde X)\in B(0,r),\quad \Vert B\Vert_{C_{t}H^{\frac{1}{2}+\delta}}\leq \frac{1}{8r},\quad \sup_{v}\frac{2\Vert B(v,\tilde X)\Vert_{C_{t}H^{\frac{1}{2}+\delta}}}{\Vert v\Vert_{C_{t}H^{\frac{1}{2}+\delta}}}\leq \frac{1}{4}.
			\end{equation}
			For the first claim of \eqref{eq;3.1;1}, fixing $ 0<T\leq  1\wedge \bar T$ small enough to make sure that $C_2T^{\frac{\delta}{\gamma}}\Vert\tilde  X\Vert^2_{C_{\bar T} W^\alpha}\leq \frac{r}{4}$, then by Lemma \ref{le;est;bilinear}, we have
			\begin{align*}
				\Vert a\Vert_{C_TH^{\frac{1}{2}+\delta}}\leq&\Vert P(\cdot)u_0\Vert_{C_TH^{\frac{1}{2}+\delta}}+\Vert B(\tilde X,\tilde X)\Vert_{C_TH^{\frac{1}{2}+\delta}}\\
				\leq& \Vert u_0\Vert_{C_TH^{\frac{1}{2}+\delta}}+C_2T^{\frac{\delta}{\gamma}}\Vert \tilde X\Vert^2_{C_{\bar T} W^\alpha}\leq r.
			\end{align*}
			For the second claim of \eqref{eq;3.1;1}, fixing $0<T\leq  1$ small enough to make sure that $C_1 T^{\frac{\gamma-1}{\gamma}}\leq\frac{1}{8r}$, then by Lemma \ref{le;est;bilinear}, we have
			\begin{align*}
				\Vert B(v,v)\Vert_{L^\infty_TH^{\frac{1}{2}+\delta}}\leq& C_1T^{\frac{\gamma-1}{\gamma}}\Vert v\Vert_{C_T H^{\frac{1}{2}+\delta}}\Vert v\Vert_{C_T H^{\frac{1}{2}+\delta}}\\
				\leq& \frac{1}{8r}\Vert v\Vert_{C_T H^{\frac{1}{2}+\delta}}^2.
			\end{align*}
			For the third claim of \eqref{eq;3.1;1}, fixing $0<T\leq  1\wedge \bar T$ small enough to make sure that  $C_3T^{\frac{\delta}{\gamma}}\Vert \tilde X\Vert_{C_{\bar T} W^\alpha}\leq\frac{1}{8}$, then by Lemma \ref{le;est;bilinear}, we have
			\begin{align*}
				\frac{\Vert B(v,\tilde X)\Vert_{C_TH^s}}{\Vert v\Vert_{C_TH^s}}\leq& C_3T^{\frac{\delta}{\gamma}}\Vert \tilde X\Vert_{C_{\bar T}W^{\alpha}}\leq\frac{1}{8}.
			\end{align*}
			Let $T^*\leq \bar T$ satisfy all the condition above, we prove our claim is hold on $[0,T^*]$. By Lemma \ref{le;Picard th}, we finish our proof.
		\end{proof}
	\end{proposition}
	\begin{proof}[\textbf{Proof of Theorem \ref{th;local result}}]
		Let $\gamma\in(\frac{4}{3},2]$, by Lemma \ref{le;regularity}, we have $X(t)\in C_{T^*} W^{\alpha}$ for some $\alpha\in(\frac{1}{6},\frac{\gamma}{2}-\frac{1}{2})$ where $X(t)$ satisfies \eqref{eq;linear evolution;noise}. By Proposition \ref{prop;local result;reg}, then we finish our proof by writing equation \eqref{eq;Burgers;white noise} as the \eqref{eq;C-M} with fixing $v=u-X$ and $\tilde X(t)=X(t)=\int_0^tP(t-s)\xi(s)\mathrm{d}s$.
	\end{proof}
	\begin{remark}
		While we prove the existence of the solution in the space $L_t^\infty H^s$, the equation's structure also allow the continuity of $v$ in time to be readily established. Observing that
		\begin{align*}
			v(t)-v(s)=(S(t)-S(s))u_0+B_{s,t}(v+X,v+X),
		\end{align*}
		where $B_{s,t}(u,u):=\int_s^tP(t-s)D(u^2)\mathrm{d}s$, Using the similar argument for proving well-posedness, we can easily get the continuous.
	\end{remark}
	To obtain a global solution, further energy estimates are required. A classical energy estimate for the heat equation (see \cite{bahouri2011fourier}) is applied to derive the corresponding estimate for the rough Burgers equation.
	\begin{lemma}\label{le;est;energy}
		Let $\gamma\in(1,2]$. Let $v$ be the solution in $\ml C([0,T];\ml S'(\mathbb{T}^d))$ of the Cauchy problem
		\begin{equation}\label{eq;heat equation;general}
			\left\{
			\begin{aligned}
				&\p_t v-\Lambda^\gamma v=f,\\
				&v|_{t=0}=v_0,\\		
			\end{aligned}
			\right.
		\end{equation}
		with $f\in L^2([0,T]; \dot H^{s-\frac{\gamma}{2}})$ and $v_0\in \dot H^s(\mathbb{T}^d)$. Then, 
		\begin{align*}
			v\in (\cap_{p=2}^\infty L^p([0,T]; \dot H^{s+\frac{\gamma}{p}}))\cap \ml C([0,T]; \dot H^s).
		\end{align*}
		Moreover we have the following estimates:
		\begin{align*}
			\Vert v(t)\Vert_{\dot H^s}^2+2\int_0^t\Vert \Lambda^\frac{\gamma}{2} v(t')\Vert_{\dot H^s}^2\mathrm{d}t'=&\Vert v_0\Vert_{\dot H^s}^2+2\int_0^t\langle f(t'),v(t')\rangle_s\mathrm{d}t',\\
			(\sum_{k\neq 0}|
			k|^{2s}(\sup_{0\leq t'\leq t}\vert \hat v(t',\xi)\vert)^2)^\frac{1}{2}\leq& \Vert v_0\Vert_{\dot H^s}+\Vert f\Vert_{L_T^2(\dot H^{s-\frac{\gamma}{2}})},\\
			\Vert v(t)\Vert_{L_T^p(\dot H^{s+\frac{\gamma}{p}})}\leq& \Vert v_0\Vert_{\dot H^s}+\Vert f\Vert_{L_T^2(\dot H^{s-\frac{\gamma}{2}})},
		\end{align*}
		with $\langle a,b\rangle_s=\int\langle \xi\rangle^{2s}\hat a(\xi)\overline{\hat b(\xi)}\mathrm{d}\xi$
		\begin{proof}
			The case of $\gamma=2$ is proved in \cite{bahouri2011fourier}, the extension to the case $\gamma\in(1,2]$ is straightforward.
		\end{proof}
	\end{lemma}
	Next, we build the well-posedness in $L^4_T\dot H^{s+\frac
		\gamma 4}$, for $0\leq s<\gamma-\frac32$. The key observation is that for $s_1\leq s_2$, we have
	\begin{equation}\label{eq;key observation}
		\Vert \p_xf \Vert_{\dot H^{s_1}}\leq \Vert \p_x f\Vert_{\dot H^{s_2}}.
	\end{equation}
	\begin{proposition}\label{prop;L4 existence}
		Let $u_0\in  \dot H^{s}$ for $s\geq0$ with $\gamma\in(\frac85,2]$, then there exists a constant $C$ such that if $T_0,\bar T,\alpha$ satisfy the following conditions
		\begin{equation}\label{eq;condition;X}
			CT_0^{\frac14}\Vert \tilde X\Vert_{C_{\bar T}W^\alpha} \leq\frac1{4}\quad for \; 0<T_0\leq \bar T \quad and\quad \alpha\geq\max(\frac{3}{2}-\frac{3\gamma}{4},s+1-\frac\gamma2),
		\end{equation}
		then there exists a positive time $T\leq T_0$ such that \eqref{eq;C-M} has a unique solution in $L_T^4([0,T];\dot H^{s+\frac\gamma 4})$. Moreover let $T_{u_0}$ denote the maximal time of existence of such a solution, then 
		\begin{itemize}
			\item {For the same constant $C$, \begin{align*}
					\Vert u_0\Vert_{ \dot H^{s}}\leq\frac1{16C}  \Longrightarrow T_{u_0}\geq T_0.
			\end{align*}}
			\item {If $T_{u_0}\leq T_0$, then \begin{equation}\label{eq;blow up criterion}
					\int_0^{T_{u_0}} \Vert u(t)\Vert_{\dot H^{s+\frac{\gamma}{4}}}^4dt=\infty.
			\end{equation}}
		\end{itemize}
		\begin{proof}
			We firstly give the local result for small initial data $u_0$. Noting that 
			$B(f,g)$ is the solution of 
			\begin{align*}
				\p_t B(f,g)-\Lambda^\gamma B(f,g)=\p_x (fg).
			\end{align*}
			Let $\gamma>\frac 85$, by Lemma \ref{le;est;energy}, \eqref{eq;key observation} and Lemma \ref{le;est;paraproduct},  we have 
			\begin{align}
				\label{eq;a;L4}\Vert P(t)u_0+B(\tilde X,\tilde X)\Vert_{L_T^4\dot H^{s+\frac{\gamma}{4}}}\leq& \Vert u_0\Vert _{\dot H^{s}}+CT^{\frac12}\Vert \tilde X\Vert_{C_TW^\alpha}^2,\\
				\label{eq;Lv;L4}2\Vert B(v,\tilde X)\Vert_{L_T^4\dot H^{s+\frac{\gamma}{4}}}\leq& CT^{\frac{1}{4}}\Vert \tilde X\Vert_{C_TW^\alpha}\Vert v\Vert_{L_T^4\dot H^{s+\frac{\gamma}{4}}},\\
				\label{eq;Bvv;L4}\Vert B(v,v)\Vert_{L_T^4\dot H^{s+\frac{\gamma}{4}}}\leq & C\Vert v\Vert_{L_T^4\dot H^{s+\frac{\gamma}{4}}}^2.
			\end{align}
			Let $r= \frac{1}{8 C}$, we know that if $\Vert u_0\Vert_{\dot H^{s}}\leq \frac1{16C}$ and 
			$$CT_0^{\frac14}\Vert \tilde X\Vert_{C_{\bar T}W^\alpha} \leq \frac1{4},\quad$$
			by Lemma \ref{le;Picard th}, there exists a unique solution of  \eqref{eq;C-M} in the ball with center $0$ and radius $\frac{1}{4C}$ in the space $L^4([0,T_0];\dot H^{s+\frac\gamma 4})$. Now we consider the case of large initial data $u_0$ in $\dot H^s$. We split $u_0$ into small part in $\dot H^s$ and a large part with compactly supported Fourier transform. For that we fix some positive real number $\rho_{u_0}$ such that 
			\begin{align*}
				(\sum_{\vert k\vert \geq \rho_{u_0}} |k|^{2s}\vert \widehat u_0(k)\vert^2 )^{\frac 12}\leq \frac 1 {32C}.
			\end{align*}
			Then by $\Vert P(t)u_0\Vert _{L_T^4\dot H^{s+\frac\gamma 4}}\leq \Vert u_0\Vert_{\dot H^s}$ and $u_0^\flat= \ml F^{-1}(1_{B(0,\rho_{u_0})}\widehat{u}_0)$, we get 
			\begin{align*}
				\Vert P(t)u_0\Vert _{L_T^4\dot H^{s+\frac\gamma 4}}\leq& \frac{1}{32C} +\Vert P(t)u^\flat_0\Vert _{L_T^4\dot H^{s+\frac\gamma 4}}\\
				\leq&\frac{1}{32C} +\rho_{u_0}^{\frac\gamma 4}\Vert P(t)u^\flat_0\Vert _{L_T^4H^{s}}\\
				\leq&\frac{1}{32C} +(\rho_{u_0}^\gamma T)^{\frac1 4}\Vert u_0\Vert_{\dot H^{s}}.
			\end{align*}
			Thus if 
			\begin{equation}\label{eq;existence time;condition}
				T\leq \left(\frac{1}{\rho_{u_0}^{\frac\gamma 4}32C\Vert u_0\Vert_{\dot H^{s}}}\right)^4\wedge T_0,
			\end{equation}
			then we have the existence of a unique solution in the ball with center $0$ and radius $\frac 1{4C}$ in the space $L_T^4\dot H^{s+\frac\gamma 4}$. Noting that if $v$ is a solution in $L^4([0,T];\dot H^{s+\frac\gamma 4})$, by Lemma \ref{le;est;energy}, we have $v\in \ml C([0,T];\dot H^{s})\cap L^2([0,T];\dot H^{s+\frac{\gamma}{2}})$
			\par Finally, we prove the blow-up criterion. Assume that we have a solution $v$ of \eqref{eq;C-M}
			on the time interval $[0,T[$ such that 
			\begin{align*}
				\int_0^T \Vert v(t)\Vert_{\dot H^{s+\frac\gamma 4}}^4dt<\infty,\quad T<T_{0}<\bar T.
			\end{align*}
			We claim that the lifespan $T_{u_0}$ of $v$ is greater than $T$. Indeed, thanks to Lemma \ref{le;est;energy}, we have 
			\begin{align*}
				\sum_{k\neq 0} |k|^{2s}(\sup_{t\in[0,T[}|\widehat v(t,k)|)^2\leq& \Vert v_0\Vert _{H^s}+\Vert \p_x (v^2+2v \tilde X+\tilde X^2)\Vert_{L_T^2\dot H^{s-\frac\gamma 2}}\\
				\leq & \Vert v_0\Vert _{H^s}+C\Vert v\Vert_{L_T^4\dot H^{s+\frac{\gamma}{4}}}^2+CT^{\frac{1}{4}}\Vert\tilde X\Vert _{C_{\bar T}W^\alpha}\Vert v\Vert _{L_T^4 \dot H^{s+\frac{\gamma}{4}}}+CT^{\frac{1}{2}}\Vert \tilde X\Vert_{C_{\bar T}W^\alpha}^2\\
				\leq&\Vert v_0\Vert_{H^s}+C\Vert v\Vert_{L_T^4\dot H^{s+\frac{\gamma}{4}}}^2+\frac14\Vert v\Vert_{L_T^4\dot H^{s+\frac{\gamma}{4}}}+\frac{1}{16C}<\infty.
			\end{align*}
			Thus, there exists a positive number $\rho$ exists such that
			\begin{align*}
				\sum_{|k|\geq \rho}|k |^{2s}(\sup_{t\in[0,T[}|\widehat v(t,k)|)^2\leq \frac{1}{32C}.
			\end{align*}
			Noting that the choise of $\rho$ is independent of $t$, by condition of \eqref{eq;existence time;condition}, we can get $T_0\geq T_{u_0}>T$ which finish our proof.
		\end{proof}
	\end{proposition}
	\begin{proposition}\label{prop;est;energy}
		Let $\gamma\in(\frac{8}{5},2]$, $\alpha\in(1-\frac{\gamma}{2},\frac{\gamma}{2}-\frac{1}{2})$ and $0<T\leq \bar T$, let $v\in L_T^4 \dot H^{\frac{\gamma}{4}}$ be the solution of \eqref{eq;C-M} on $[0,T]$ for $u_0\in L^2$ with $\tilde X\in C_{\bar T}W^\alpha$. Then for any $t\in[0,T]$, we have the following energy estimate
		\begin{equation}\label{eq;est;energy;Burgers}
			\begin{aligned}
				\Vert v\Vert_{L^2}^2+2(1-\nu)\int_0^t \Vert v(t')\Vert_{H^\frac{\gamma}{2}}^2\mathrm{d}t'\leq&\Vert u_0\Vert_{L^2}^2e^{C_\nu(1+\Vert \tilde X\Vert_{C_{\bar T}W^\alpha}^{\frac{2\gamma}{\gamma-1}})t}\\
				&+C_\nu\Vert \tilde X\Vert_{C_{\bar T}W^\alpha}^4\int_0^t e^{C_\nu(1+\Vert \tilde X\Vert_{C_{\bar T}W^\alpha}^{\frac{2\gamma}{\gamma-1}})(t-t')} \mathrm{d}t'.
			\end{aligned}
		\end{equation} 
		for some small $\nu<1$.
		\begin{proof}
			Noting that $v\in L_T^4 \dot H^{\frac{\gamma}{4}}$, we have $ f=\p_x(v+\tilde X)^2\in L_T^2\dot H^{-\frac{\gamma}{2}}$ for $T\leq\bar T$. Observing that $v$ satisfies \eqref{eq;heat equation;general}, then by Lemma \ref{le;est;energy} for $s=0$, we have $v\in L_T^2\dot H^{\frac{\gamma}{2}}\cap \ml C([0,T];L^2)$. Then by H\"older inequality, we have 
			\begin{equation}\label{eq;v3}
				\int_0^T\langle \p_x v^2,v\rangle_{0}\mathrm dt\leq \Vert v\Vert_{L_T^2\dot H^\frac{\gamma}{2}}\Vert v\Vert_{L_T^4\dot H^{\frac{\gamma}{4}}}^2.
			\end{equation}
			for any $\gamma>\frac{3}{2}$. Using the estimates in Lemma \ref{le;est;energy} and integral by parts, we have
			\begin{align*}
				\Vert v(t)\Vert_{L^2}^2+2\int_0^t \Vert v(t')\Vert_{H^\frac{\gamma}{2}}^2\mathrm{d}t'\leq& \Vert u_0\Vert_{L^2}^2+2\int_0^t\Vert v(t')\Vert_{L^2}^2\mathrm{d}t'\\
				&+2\int_0^t\langle v(t'),2\p_x(v(t')\tilde X(t'))+\p_x(\tilde X^2(t'))\rangle_{0}\mathrm{d}t'.
			\end{align*}
			By H\"older inequality, Young's inequality, Lemma \ref{le;est;paraproduct} and Lemma \ref{le;regularity}, there exists $\alpha\in(1-\frac{\gamma}{2},\frac{\gamma}{2}-\frac{1}{2})$ such that
			\begin{align*}
				\langle v,2\p_x(v\tilde X)+\p_x(\tilde X^2)\rangle_{0}\leq& 2\Vert v\Vert_{H^\frac\gamma2}\Vert v\tilde X\Vert_{H^{1-\frac{\gamma}{2}}}+\Vert v\Vert_{H^\frac\gamma2}\Vert \tilde X^2\Vert_{H^{1-\frac{\gamma}{2}}}\\
				\leq&2\Vert v\Vert_{H^\frac\gamma2}^{1+\frac{1}{\gamma}}\Vert v\Vert_{L^2}^{1-\frac{1}{\gamma}}\Vert \tilde X\Vert_{W^\alpha}+\Vert v\Vert_{H^\frac\gamma2}\Vert \tilde X\Vert_{W^\alpha}^2\\
				\leq & \frac{\nu}{2} \Vert v\Vert_{H^\frac\gamma2}^{2}+C_\nu(\Vert v\Vert_{L^2}^{2}\Vert \tilde X\Vert^{\frac{2\gamma}{\gamma-1}}_{W^\alpha}+\Vert \tilde X\Vert_{W^\alpha}^4).
			\end{align*}
			For $0\leq t\leq T$, we have 
			\begin{equation}
				\begin{aligned}
					\Vert v(t)\Vert_{L^2}^2+2\int_0^t \Vert v(t')\Vert_{H^\frac{\gamma}{2}}^2\mathrm{d}t'\leq& \Vert u_0\Vert_{L^2}^2+\nu \int_0^t \Vert v(t')\Vert_{H^\frac{\gamma}{2}}^2\mathrm{d}t'+C_\nu\int_0^t (1+\Vert \tilde X(t')\Vert_{W^\alpha}^{\frac{2\gamma}{\gamma-1}})\Vert v(t')\Vert_{L^2}^2\mathrm{d}t'\\
					&+C_\nu\int_0^t\Vert \tilde X(t')\Vert_{W^{\alpha}}^4\mathrm{d}t'.	
				\end{aligned}
			\end{equation}
			Observing that $\Vert\tilde X\Vert_{C_tW^\alpha}\leq \Vert \tilde X\Vert_{C_{\bar T}W^\alpha}$, for $\nu<1$, by Gronwall's inequality, the proposition is proved.
		\end{proof}
		\begin{remark}
			For higher dimensions, estimate \eqref{eq;v3} may fail, since
			\[
			\|v^2\|_{\dot H^{1-\frac{\gamma}{2}}} \lesssim \|v\|_{\dot H^{\frac{\gamma}{4}}}^2
			\]
			holds only under the condition $\frac{\gamma}{2} - \frac{d}{2} \geq 1 - \frac{\gamma}{2}$, which implies $\gamma\geq 2$ for $d \geq 2$. To overcome this difficulty, a compactness method is required—a standard approach in the analysis of the Navier–Stokes equations (see, e.g., \cite{hairer2024global}, \cite{bahouri2011fourier}, and \cite{bedrossian2022mathematical}).
		\end{remark}
	\end{proposition}
	\begin{proof}[\textbf{Proof of Theorem \ref{th;global result}}]
		\par Noting that $X$ satisfy \eqref{eq;linear evolution;noise}, by Lemma \ref{le;regularity}, we have $X\in C_{\bar T}W^\alpha$ for any $\alpha<\frac{\gamma}{2}-\frac{1}{2}$ and $\bar T>0$. Let $\tilde X=X$, for $\gamma\in]\frac{8}{5},2]$, then we can find $T_0,\alpha$ satisfy the following conditions
		\begin{equation}
			CT_0^{\frac14}\Vert X\Vert_{C_{\bar T}W^\alpha} \leq\frac1{4}\quad for \; 0<T_0\leq \bar T \quad and\quad \alpha\geq\max(\frac{3}{2}-\frac{3\gamma}{4},1-\frac\gamma2),
		\end{equation}
		Recall $u_0\in L^2$ with zero-mean, let $v$ be the solution satisfying Proposition \ref{prop;L4 existence}. To obtain the global solution, we proceed with our induction argument. Firstly
		We prove that there exists a local solution on $[0,T_0]$. Indeed, by energy estimate \eqref{eq;est;energy;Burgers} and the following blow-up criterion that if $T_{u_0}\leq T_0$, then
		$$\int_0^{T_{u_0}} \Vert v(t)\Vert^4_{\dot H^{\frac\gamma 4}}dt=\infty .$$
		By interpolation 
		\begin{equation}
			\Vert v\Vert_{L_T^4\dot H^{\frac{\gamma}{4}}}^4\lesssim \Vert v\Vert_{L_T^\infty L^2}^2\Vert v\Vert_{L_T^2\dot H^\frac{\gamma}{2}}^2,
		\end{equation}
		if $T_{u_0}\leq T_0$, then we have
		\begin{align*}
			\forall\; T<T_{u_0}, \quad\int_0^{T} \Vert v(t)\Vert^4_{\dot H^{\frac\gamma 4}}dt\leq C(\Vert u_0\Vert_{L^2}, \Vert X\Vert_{C_{\bar T }W^\alpha}, T_0)<\infty.
		\end{align*}
		with uniform upper bound. By contradiction, we have $T_{u_0}> T_0$, which shows that $v$ is the solution on $[0,T_0]$. Consider equation \eqref{eq;C-M} with initial data $v(T_0)$. Setting $ v^{(1)}(t)=v(T_0+t)$, $ X^{(1)}(t)=X(T_0+t)$, we find that $v^{(1)}(t)$ satisfies the equation
		\begin{equation}
			\p_t v^{(1)}(t)-\Lambda^\gamma  v^{(1)}(t)=\p_x((v^{(1)}+ X^{(1)})^2),\quad  v^{(1)}(0)=v({T_0}).
		\end{equation}
		By a similar argument, we can also get a solution $v^{(1)}$ of equation on $[0,T_{v(T_0)}\wedge T_0[$ , since 
		$$\Vert  X^{(1)}\Vert_{C_tW^{\alpha}}=\Vert X\Vert_{C_{[T_0,T_0+t]}W^{\alpha}}\leq\Vert X\Vert_{C_{\bar T}W^{\alpha}},\quad for\: t\in[0,\bar T-T_0].$$  Then using the similar argument as above implies that $v^{(1)}(t)$ is the solution on $[0,T_0]$.
		Repeating such process, we can get $v^{(n)},X^{(n)}$ satisfying 
		\begin{equation}
			\Vert X^{(n)}\Vert_{C_{t}W^\alpha}\leq\Vert X\Vert_{C_{\bar T}W^{\alpha}},\quad v^{(n)}(t)=v^{(n-1)}(T_0+t),\quad for\: t\in[0,\bar T-nT_0],
		\end{equation}
		where $v^{(n)}$ on $[0,T_0]$ is also a solution of \eqref{eq;C-M}. Setting $v:=v^{(0)}$ on $[0,T_0]$, $v^{new}:=v^{(i)}$ on $[iT_0,(i+1)T_0]$ for $i\leq n$, $v^{new}$ is a solution in $C([0,(n+1)T_0];L^2(\mathbb{T}))\cap L^2([0,(n+1)T_0];H^\frac{\gamma}{2}(\mathbb{T}))$.  Choosing $n=[\frac{\bar T}{T_0}]+1$, we get $v\in C([0,\bar T];L^2(\mathbb{T}))\cap L^2([0,\bar T];H^\frac{\gamma}{2}(\mathbb{T}))$. Since $\bar T$ is arbitrary, we obtain $v$ as a weak global solution of \eqref{eq;C-M}. By the proof of Proposition \ref{prop;L4 existence}, it's unique. Since $X$ satisfies \eqref{eq;linear evolution;noise}, we prove that $v$ is the weak global solution of \eqref{eq;Burgers;white noise}.
	\end{proof}
	\subsection{Application to DP equation}
	This section demonstrates the application of the aforementioned theory to the study of a class of shallow water wave equations. Taking the DP equation as an example, similar as the equation \eqref{eq;C-M}, we consider the difference equation of type of DP.
	\begin{equation}\label{eq;C-M;DP}
		\p_t v-\Lambda^\gamma v=\frac{1+3(1-\p_x^2)^{-1}}{2}\p_x(v^2+2v\tilde X+\tilde X^2),\quad v(0)=u_0.
	\end{equation}
	The only difference between \eqref{eq;C-M;DP} and \eqref{eq;C-M} lies on $\frac{3}{2}(1-\partial_x^2)^{-1}\partial_x(u^2)$. Note that the operator $(1-\partial_x^2)^{-1}$ is an `stable' operator in $H^s$ since $$\Vert (1-\partial_x^2)^{-1}f\Vert_{H^s}\leq \Vert f\Vert_{H^{s-2}},\quad for \: s\in \m R.$$
	Hence, we can easily get the local result for rough DP equation. For the global result, we need the following energy estimate for DP equation.
	\begin{proposition}\label{prop;est;energy;DP}
		For $\gamma\in(\frac{3}{2},2]$, $\alpha\in(1-\frac\gamma 2,\frac{\gamma}{2}-\frac{1}{2})$ and $0\leq T\leq \bar T<\infty$. Let $v$ be a smooth solution of \eqref{eq;C-M;DP} on $[0,T]$ for $u_0\in L^2$ and $\tilde X(t)\in C_{\bar T}W^{\alpha}$, then we have the following energy estimate
		\begin{equation}\label{eq;est;energy;DP}
			\begin{aligned}
				\Vert v\Vert_{L^2}^2+(8-8\nu)\int_0^t \Vert v\Vert_{H^\frac{\gamma}{2}}^2\mathrm{d}\tau\leq&16\Vert u_0\Vert_{L^2}^2e^{C_\nu(1+ \Vert \tilde X\Vert_{C_{\bar T}W^\alpha}^{\frac{2\gamma}{\gamma-1}})t}\\
				&+C_\nu\Vert \tilde X\Vert_{C_{\bar T}W^\alpha}^4\int_0^t e^{C_\nu(1+\Vert \tilde X\Vert_{C_{\bar T}W^\alpha}^{\frac{2\gamma}{\gamma-1}})(t-t')} \mathrm{d}t'.
			\end{aligned}
		\end{equation}
		for some small $\nu<1$.
		\begin{proof}
			Let $n=(1-\p_x^2)v$ and $w=(4-\p_x^2)^{-1}v$, similar as the proof  of Lemma 4.1 in \cite{chen2024probabilistic}, by \eqref{eq;C-M;DP}, we have
			\begin{align*}
				\int_{\mathbb{T}}v\p_xnw+3\p_xvnw\mathrm{d}x=0.
			\end{align*}
			Then we can derive it to the following equation
			\begin{align*}
				\p_t\int_{\mathbb{T}}nw\mathrm{d}x=&2\int_{\mathbb{T}}\langle \p_x\rangle^2\p_tvw\mathrm{d}x\\
				=&2\int_{\mathbb{T}}\langle \p_x\rangle^2\Lambda^\gamma vw\mathrm{d}x+2\int_{\mathbb{T}}\langle \p_x\rangle^2\p_x(2v\tilde X+\tilde X^2)w\mathrm{d}x+3\int_{\mathbb{T}}\p_x(2v\tilde X+\tilde X^2)w\mathrm{d}x.
			\end{align*}
			By Parseval formula, H\"older inequality and interpolation inequality, we have
			\begin{equation}\label{eq;est;DP;1}
				\begin{aligned}
					\int_{\mathbb{T}}\langle \p_x\rangle^2\p_x(v\tilde X)w\mathrm{d}x=&\int_{\mathbb{T}}\frac{1+\vert \xi\vert^2}{4+\vert \xi\vert^2}\ml F(\p_x (v\tilde X))\ml F(v)\mathrm{d}x\\
					\leq &\Vert v\Vert_{H^{\frac{\gamma}{2}}}\Vert v\tilde X\Vert_{H^{1-\frac{\gamma}{2}}}\\
					\leq &\Vert v\Vert_{H^\frac{\gamma}{2}}^{1+\frac{1}{\gamma}}\Vert v\Vert_{L^2}^{1-\frac{1}{\gamma}}\Vert \tilde X\Vert_{\ml C^\alpha}\\
					\leq &\frac{\nu}{6} \Vert v\Vert_{H^\frac{\gamma}{2}}^{2}+C_\nu\Vert v\Vert_{L^2}^2 \Vert \tilde X\Vert_{W^\alpha}^{\frac{2\gamma}{\gamma-1}},
				\end{aligned}
			\end{equation}
			where we use the fact $\alpha>{1-\frac{\gamma}{2}}$. Similarly, by Lemma \ref{le;est;paraproduct}, we have
			\begin{align}
				\int_{\mathbb{T}}\langle \p_x\rangle^2\p_x\tilde X^2w\mathrm{d}x\leq&\frac{\nu}{6}  \Vert v\Vert_{H^{\frac{\gamma}{2}}}^2+C_\nu\Vert\tilde X\Vert_{W^\alpha}^4,\label{eq;est;DP;2}\\
				\int_{\mathbb{T}}\p_x(2v\tilde X+\tilde X^2)w\mathrm{d}x\leq& \frac{\nu}{3}  \Vert v\Vert_{H^\frac{\gamma}{2}}^{2}+C_\nu\Vert v\Vert_{L^2}^2 \Vert \tilde X\Vert_{W^\alpha}^{\frac{2\gamma}{\gamma-1}}+C_\nu\Vert \tilde X\Vert_{W^\alpha}^4\label{eq;est;DP;3}.
			\end{align}
			Combining above estimates, we have 
			\begin{equation}\label{eq;est;DP;4}
				\begin{aligned}
					\int_{\m T}n(t)w(t)\mathrm{d}x&+(2-2\nu)\int_0^t \Vert v(t')\Vert_{H^{\frac{\gamma}{2}}}^2\mathrm{d}t'\\
					\leq&\int_{\m T}n(0)w(0)\mathrm{d}x+C_{\nu}\int_0^t\Vert v(t')\Vert_{L^2}^2 \Vert \tilde X(t')\Vert_{W^\alpha}^{\frac{2\gamma}{\gamma-1}}\mathrm{d}t'+C_{\nu}\int_0^t\Vert \tilde X(t')\Vert_{W^\alpha}^4\mathrm{d}t'.
				\end{aligned}
			\end{equation}
			Noting that 
			\begin{equation}
				\Vert v(t)\Vert_{L^2}^2= \Vert \widehat v(t)\Vert_{L^2}^2\leq 4\int_{\m T}\frac{1+\xi^2}{4+\xi^2} \vert \widehat v(t)\vert^2\mathrm{d}x=4\int_{\m T}n(t)w(t)\mathrm{d}x\leq 4\Vert v(t)\Vert_{L^2}^2.
			\end{equation}
			Particularly, $\int_{\m T}n(0)w(0)\mathrm{d}x\leq 4\Vert u_0\Vert_{L^2}^2$, by \eqref{eq;est;DP;4} we have 
			\begin{equation}
				\Vert v(t)\Vert_{L^2}^2+(8-8\nu)\int_0^t \Vert v(t')\Vert_{H^{\frac{\gamma}{2}}}^2\mathrm{d}t'\leq 16\Vert u_0\Vert_{L^2}^2+C_{\nu}\int_0^t\Vert v(t')\Vert_{L^2}^2 \Vert \tilde X(t')\Vert_{W^\alpha}^{\frac{2\gamma}{\gamma-1}}\mathrm{d}t'+C_{\nu}\int_0^t\Vert \tilde X(t')\Vert_{W^\alpha}^4\mathrm{d}t'
			\end{equation} 
			for $\nu <1$ and $t\in[0,T]$. By Gronwall's inequality, we finish our proof.
		\end{proof}
	\end{proposition}
	
	Using a similar argument as in the proof of Theorem \ref{th;global result}, we finish the proof of Remark \ref{rem;result;DP}. 
	
	\section{Result in the range of $\gamma\in(\frac{5}{4},\frac{4}{3}]$}\label{sec;Para-controlled}
	\subsection{Regularity analysis}\label{subsec;critical analysis}
	This section is devoted to isolating a remainder term of Sobolev regularity $\frac{1}{2}+$ under the weaker dissipation condition $\gamma\in(\frac{5}{4},\frac{4}{3}]$. We will achieve this by means of the paracontrolled method, introduced in \cite{gubinelli2015paracontrolled,gubinelli2019singular}, a technique especially well‑suited for analyzing the regularity structure of solutions to singular SPDEs.
	\par Our strategy as follows. Recall the difference equation 
	\begin{align*}
		\p_t v-\Lambda^\gamma v=\ml L(v)=&\ml L(u-X)\\
		=&\p_x(u^2)+\xi-\xi\\
		=&\p_x(v^2)+\p_x(X^2)+2\p_x(vX).
	\end{align*}
	Assume that $v$ gains the critical regularity $H^{\frac{1}{2}+}$. By the critical analysis, $\p_x(v^2)$, $\p_x(X^2)$ and $\p_x(vX)$ must gain $H^{\frac{1}{2}-\gamma+}$-order regularity at least. But it's impossible in the weaker dissipation conditions $\gamma\in(\frac{5}{4},\frac{4}{3}]$  since $\p_x(X^2)$ and $\p_x(vX)$ gain $\alpha-1$-order regularity at most. To overcome this difficult, our strategy is to decompose such terms in `` higher-regularity-term" and "lower- regularity-term" by para-product decomposition \eqref{eq;decomposition; VX XX}.  Correspondingly, we write
	\begin{equation}
		v=w+u^\sharp,
	\end{equation} where $\ml Lu^\sharp$ corresponds to the  "higher-regularity-term", while $\ml L w$ is designed to absorb the "lower- regularity-term".
	\par  Next, our work is carried out in the framework of the space \(W^s = \mathcal{C}^s \cap H^s\). We introduce a method called \textbf{"para-controlled solution"}, which was developed in \cite{gubinelli2015paracontrolled,gubinelli2017kpz} to denote by $u^\sharp$ the part of $u$ that gains $W^{\frac{1}{2}+}$ regularity. To get a higher order term in $W^{\frac{1}{2}+}$. Let distribution $u^\sharp$ satisfy equation
	\begin{equation}\label{eq; u-sharp}
		\ml Lu^\sharp=\p_x(u^\sharp)^2+R(X,u^\sharp,u),\quad u^{\sharp}(0)=u_0\in W^{\frac{1}{2}+}.
	\end{equation} 
	where $R(X,u^\sharp,u)$ is to be determined. Assume that $R(X,u^\sharp,u)\in W^{\frac{1}{2}-\gamma+}$. Then by regularity theory, $u^\sharp$ belongs to $W^{\frac{1}{2}+}$. Fix $u=X+w+u^\sharp$, then we have
	\begin{align*}
		\ml L(u^\sharp)=&\ml L(u-X-w)\\
		=&\p_x(u^2)+\xi-\xi-\ml Lw\\
		=&\p_x(w^2)+\p_x(X^2)+2\p_x(wX)+2\p_x(wu^\sharp)+2\p_x(Xu^\sharp)+(\p_x((u^\sharp)^2)-\ml Lw).
	\end{align*}
	Since $u^\sharp$ possesses sufficient regularity, the product $u^\sharp\cdot w$ preserve regularity of $w$. The problematic terms are $\p_x(w\cdot X)$ and $\p_x(X^2)$, as they can gain at most $\alpha-1$-order regularity which below the required  $\frac{1}{2}-\gamma+$ in weak dissipation conditions $\gamma\in(\frac{5}{4},\frac{4}{3}]$. Using para-product decomposition \eqref{eq;decomposition; VX XX} and the fact that $X\in C_T\ml C^\alpha$, we observe that $\p_x(X\circ X)$, $\p_x(X\prec w)$ and $\p_x (X\circ w)$ all belong to $H^{\frac{1}{2}-\gamma}$. The main difficulty therefore lies in the terms $\p_x(v\prec X)$ and $\p_x(X\prec X)$.
	\par To handle this, we assume that $v$ possesses a suitable structure that can absorb the problematic terms. Motivated by the para-controlled ansatz, assuming that  $\ml L\ml Q=\p_x X, \ml Q(0)=0$, we write 
	\begin{equation}
		w=u'\pprec\ml Q,\quad u^{\ml Q}:=w+u^\sharp,
	\end{equation} 
	with the assumption $u'\in \ml W^{\frac{1}{8}+}(T)$, $\ml Q\in \ml W^{\frac{3}{8}+}$ and $u^\sharp\in \ml W^{\frac{1}{2}+}(T)$ where $\ml W^s(T)$ is defined in Section \ref{sec;basic tools}. By Lemma \ref{le;est;paraproduct;modifed}, this implies $w\in \ml W^{\frac{3}{8}+}(T)$. For $\gamma>\frac{5}{4}$, the regularity analysis proceeds as follows
	\begin{align*}
		&2[\p_x(X\prec X)-X\prec\p_x X]+\p_x(w^2)+2\p_x(wu^\sharp)+\p_x(X\circ X)\in C_TW^{-\frac{3}{4}+},\\
		&2\p_x(X\circ u^{\ml Q})\in C_TW^{-\frac{1}{2}+},\quad 2[u'\pprec \p_xX-u'\prec \p_xX]\in C_TW^{-\frac{3}{4}+},\\
		&2\p_x(X\prec u^{\ml Q})\in C_TW^{-\frac{5}{8}+},\quad u'\pprec\ml L\ml Q-\ml L(u'\pprec \ml Q)\in C_TW^{\frac{1}{2}-\gamma+}.
	\end{align*}
	Therefore, we can determine $R(X,u^\sharp,u)$ as
	\begin{align*}
		R(X,u^\sharp,u):=&\p_x(w^2)+2\p_x(wu^\sharp)+\p_x(X\circ X)+2[\p_x(X\prec X)-X\prec\p_x X]\\
		&+2\p_x(X\prec u^\ml Q+X\circ u^\ml Q+[\p_x( u^\ml Q\prec X)- u^\ml Q\prec \p_xX ] \\
		&+u'\pprec\ml L\ml Q-\ml L(u'\pprec \ml Q)+[(2X+2u^{\ml Q})\prec \p_xX-u'\pprec \p_xX].
	\end{align*}
	Here we see that the purpose of using modified paraproduct operator $f\pprec g$ is to ensure that the commutator $\ml L(u'\pprec \ml Q)-u'\pprec\ml L\ml Q$ aligns with our regularity framework. Let $u'=2X+2u^\ml Q$, we counter the rough term
	\begin{align*}
		u'\prec \ml LQ=&2X\prec \p_xX+2u^{\ml Q}\prec\p_x X,
	\end{align*}
	where the last term has higher regularity. By equation \eqref{eq; u-sharp}, we construct a paracontrolled solution if we can verify $u^\sharp$ has enough regularity. In fact, by regularity theory, for $\gamma>\frac{5}{4}$, \eqref{eq; u-sharp} derive the regularity of $u^\sharp$ such as 
	$$u^\sharp\in \ml W^{\frac{1}{2}+\delta}(T),$$
	which we finish our ansatz setting.
	\subsection{Construction of para-controlled solution}
	\par We frame this problem as solving a system of equations of $(u',u^\sharp) $. Given an equation of $(u',u^\sharp)$ driven by $(\tilde X,u_0)$
	\begin{equation}\label{eq;paracontrolled solution}
		\left\{
		\begin{aligned}
			&u'=2\tilde X+2u^{\ml Q}\\
			&\ml Lu^\sharp=[\p_x(\tilde X^2)-2\tilde X\prec \p_x\tilde X]+\p_x(u^{\ml Q})^2+2[\p_x(u^\ml Q \tilde X)-u^\ml Q\prec \p_x\tilde X]+[u'\prec \p_x\tilde X-\ml L(u'\pprec\ml Q)],
		\end{aligned}
		\right.
	\end{equation}
	with $u^{\sharp}(0)=u_0$. We have the following proposition
	\begin{proposition}\label{prop;local result;1}
		Let $\alpha\in(\frac{1}{8},\frac{1}{6})$, $\gamma\in(\frac{5}{4},\frac{4}{3}]$ and $u_0\in W^{\frac{1}{2}+\delta}$ for $0<\delta\leq \min(\gamma-\frac{5}{4},\alpha-\frac{1}{8})$. If $\tilde X\in \ml W^{\alpha}(\bar T)$ for some $\bar T> 0$, then there exists a unique solution $(u',u^\sharp)$ of equation \eqref{eq;paracontrolled solution}  with positive $T^*$  satisfied $(u',u^\sharp)\in \ml W^{\frac{1}{8}+\delta}(T^*)\times \ml W^{\frac{1}{2}+\delta}(T^*)$.
		\begin{proof}
			Firstly, let $u^\sharp(0)=u_0$ and $f(u^\sharp,u')=u^{\ml Q}=u'\pprec Q+u^\sharp$, we can rewrite \eqref{eq;paracontrolled solution} in mild form
			\begin{align*}
				u'=2\tilde X +2u^{\ml Q},\quad \quad u^\sharp=a+g(u')+\ml B(f(u^\sharp,u'),f(u^\sharp,u'))+\ml G(f(u^\sharp,u'))
			\end{align*}
			and
			\begin{align*}
				\ml L a=\p_x(\tilde X^2)-2\tilde X\prec \p_x\tilde X,&\quad a(0)=u_0, \\
				\ml L(\ml G(f(u^\sharp,u')))=2[\p_x(u^\ml Q \tilde X)-u^\ml Q\prec \p_x\tilde X],&\quad \ml G(f(u^\sharp,u'))(0)=0,\\
				\ml L{\ml B}(f(u^\sharp,u'),f(u^\sharp,u'))=\p_x(u^{\ml Q})^2,&\quad {\ml B}(f(u^\sharp,u')(0)=0,\\
				\ml L g(u')=[u'\pprec \p_x\tilde X-\ml L(u'\pprec\ml Q)]+[u'\prec \p_x\tilde X-u'\pprec \p_x\tilde X],&\quad g(u')(0)=0.
			\end{align*}
			Define $\ml E_{T}$ as follows
			\begin{align*}
				\ml E_T:=\{(u,v)\in \ml W^{\frac{1}{2}+\delta}(T)\times \ml W^{\frac{1}{8}+\delta}(T): 2\Vert u\Vert_{\ml W^{\frac{1}{2}+\delta}(T)}+\Vert v\Vert_{\ml W^{\frac{1}{8}+\delta}(T)}<\infty\}.
			\end{align*}
			Our goal is to find a $T^*\leq 1\wedge \bar T$ such that the conditions of Lemma \ref{le;Picard th;specific} hold in $\ml E_{T^*}$. Assume that $T\leq 1\wedge \bar T$, by Lemma \ref{le;est;paraproduct;modifed}, for $\epsilon=\frac{1}{8}+\delta$, $\rho=\gamma-\delta$ and $\delta\leq \min(\gamma-\frac{5}{4},\alpha-\frac{1}{8})$, we have
			\begin{align*}
				\Vert f(u^\sharp,u')\Vert_{\ml W^{\frac{3}{8}+\delta}(T)}\leq& \Vert u^\sharp\Vert_{\ml W^{\frac{1}{2}+\delta}(T)}+CT^{\frac{\delta}{\gamma}}\Vert u'\Vert_{\ml W^{\frac{1}{8}+\delta}(T)}(\Vert \ml Q\Vert_{C_T\ml C^{\frac{3}{8}+\delta}}+\Vert \p_x\tilde X\Vert_{C_T\ml C^{\frac{3}{8}-\gamma+2\delta}}),\\
				\leq& \Vert u^\sharp\Vert_{\ml W^{\frac{1}{2}+\delta}(T)}+CT^{\frac{\delta}{\gamma}}\Vert u'\Vert_{\ml W^{\frac{1}{8}+\delta}(T)}\Vert  \tilde X\Vert_{\ml W^{\alpha}(\bar T)}.
			\end{align*}
			Fixing $T$ small enough, we can make sure that $CT^{1-\frac{\delta}{\gamma}}\Vert  \tilde X\Vert_{\ml W^{\alpha}(\bar T)}<\frac{1}{30}$.
			\par For fixed $r=16(1\vee C)\max(\Vert u_0\Vert_{ W^{\frac{1}{2}+\delta}}, \Vert \tilde X\Vert_{\ml W^{\alpha}(\bar T)})$, by Lemma \ref{le;heat flow} and Lemma \ref{le;commutater;D}, we can choose $T\leq 1\wedge \bar T$ small enough such that
			\begin{align*}
				\Vert (a,\tilde X)\Vert_{\ml E_{T}}
				\leq& \Vert \tilde X\Vert_{\ml W^{\alpha}(T)} +2C\Vert u_0\Vert_{ W^{\frac{1}{2}+\delta}}+2CT^{\frac{\delta}{\gamma}}\Vert \ml La\Vert_{ C_TW^{\frac{1}{2}-\gamma+2\delta}}\\
				\leq& \frac{r}{16}+\frac{r}{8}+\frac{r}{16}\leq \frac{r}{4}.
			\end{align*}
			By Lemma \ref{le;commutator;modfied paraproduct, L} and Lemma \ref{le;commutater;D}, we have a similar argument as above. We can choose $T\leq 1\wedge \bar T$ small enough such that
			\begin{align*}
				\Vert g(u')\Vert_{ \ml W^{\frac{1}{2}+\delta}(T)}\leq CT^{\frac{\delta}{\gamma}}\Vert \ml Lg(u')\Vert_{C_TW^{\frac{1}{2}-\gamma+2\delta}}\leq C& T^{\frac{\delta}{\gamma}}\Vert u'\Vert_{\ml W^{\frac{1}{8}+\delta}(T)}(\Vert \ml Q\Vert_{C_{{T}}W^{\frac{3}{8}+\delta}}+\Vert \p_x\tilde  X\Vert_{C_TW^{\frac{3}{8}-\gamma+\delta}})\\
				\leq &C T^{\frac{\delta}{\gamma}}\Vert u'\Vert_{\ml W^{\frac{1}{8}+\delta}(T)}\Vert \tilde X\Vert_{C_{\bar T}W^{\alpha}}
				\leq \frac{1}{16}\Vert u'\Vert_{\ml W^{\frac{1}{8}+\delta}(T)}.
			\end{align*}
			and 
			\begin{align*}
				\Vert \ml G( f(u^\sharp,u') )\Vert_{\ml W^{\frac{1}{2}+\delta}(T)}\leq &C T^{\frac{\delta}{\gamma}}\Vert \ml L\ml G( f(u^\sharp,u'))\Vert_{C_T W^{\frac{1}{2}-\gamma+2\delta}} \\
				\leq&CT^{\frac{\delta}{\gamma}}\Vert f(u^\sharp,u')\Vert_{\ml W^{\frac{3}{8}+\delta}}\Vert \tilde X\Vert_{\ml W^{\alpha}(\bar T)}\\
				\leq &\frac{1}{16}\Vert f(u^\sharp,u')\Vert_{\ml W^{\frac{3}{8}+\delta}}.
			\end{align*}
			Finally, let $Y=\ml W^{\frac{3}{8}+\delta}(T)$, we have $\ml W^{\frac12+\delta}(T)\subset\ml W^{\frac{3}{8}+\delta}(T)\subset\ml W^{\frac18+\delta}(T)$, Choosing $T$ small enough such that 
			\begin{align*}
				\Vert \ml B(f(u^\sharp,u'),f(u^\sharp,u'))\Vert_{\ml W^{\frac{1}{2}+\delta}(T)}\leq& C T^{\frac{\delta}{\gamma}}\Vert \ml L\ml B( f(u^\sharp,u') )\Vert_{C_T W^{\frac{1}{2}-\gamma+2\delta}} \\
				\leq &CT^{\frac{\delta}{\gamma}}\Vert \p_x( f(u^\sharp,u')^2 )\Vert_{C_T W^{\frac{1}{2}-\gamma+2\delta}}\\
				\leq &C T^{\frac{\delta}{\gamma}}\Vert f(u^\sharp,u')\Vert^2_{\ml W^{\frac{3}{8}+\delta}(T)}\\
				\leq &\frac{1}{16r}\Vert f(u^\sharp,u')\Vert^2_{\ml W^{\frac{3}{8}+\delta}(T)}.
			\end{align*}
			Setting $T^*$ to satisfy all the conditions above, by Lemma \ref{le;Picard th;specific}, there exists $(u^\sharp,u')\in\ml E_{T^*}$ for which we prove our proposition.
		\end{proof}
	\end{proposition}
	\begin{proof}[\textbf{Proof of Theorem \ref{th;paracontrolled}}]
		Let $\tilde X(t)=X(t)$, where $X(t)$ satisfies equation \eqref{eq;linear evolution;noise}. By Remark \ref{rem;regularity}, we have $\tilde X\in \ml W^{\alpha}({\bar T})$ for any fixed $\bar T>0$ and $\alpha\in (\frac{1}{8},\frac{\gamma}{2}-\frac{1}{2})$. Then by Proposition \ref{prop;local result;1}, there exists $(u',u^\sharp)\in C_{T^*}W^{\frac{1}{8}+\delta}\times C_{T^*}W^{\frac{1}{2}+\delta}$ satisfying \eqref{eq;paracontrolled solution} for $\delta\leq \min(\gamma-\frac{5}{4},\alpha-\frac{1}{8})<\frac{\gamma}{2}-\frac{5}{8}$ and $0<T^*\leq\bar T$. It's not difficult to check that $u=X+u'\pprec \ml Q+u^{\sharp}$ is the solution of \eqref{eq;Burgers;white noise} by the analysis in Section \ref{subsec;critical analysis}. Since $u\in C_{T^*}W^{\alpha}$, we finish the proof of Theorem \ref{th;paracontrolled}.
	\end{proof}
	\section{ Discussion on the range of $\gamma$}\label{sec;conjecture}
	In the final section, we propose a conjecture regarding the admissible range of the dissipation parameter $\gamma$ for the rough or singular Burgers equation driven by $\vert D\vert^{\beta}\xi$ by the method in \cite{deng2022random}.
	\begin{equation}\label{eq;Burgers;white noise;singular}
		u_t-\Lambda^\gamma u=\p_x(u^2)+\vert D\vert^{\beta}\xi,
	\end{equation}
	where $\beta\geq0$. For $\zeta=\vert D\vert^{\beta}\xi$, we fix a dyadic number $N\neq0 $, and set $\zeta^N:=\ml F^{-1}(1_{\vert k\vert\sim N}\ml F\xi(k))$, which means that the frequency of $\xi$ is focus on $\frac{1}{2}N\leq \vert k\vert\leq N$. Then the Fourier mode of the linear evolution $X:=\int_{0}^te^{ (t-s)\Lambda^\gamma}\zeta^N\mathrm{d}s$ can be written as 
	$$X=N^{\beta-\frac{\gamma}{2}}\sum_{\vert k\vert\sim N}G_k(t)e^{ikx},\quad G_k(t)=N^{\frac{\gamma}{2}}\int_0^{t}e^{ -(t-s)\vert k\vert^\gamma}\xi_k(s)\mathrm{d}s,$$
	It's not difficult to calculate that $G_k(t)$ from a collection of independent Gaussian variable with $\m E[\vert G_k(t)\vert^2]\sim 1$, then we can get $X$ belong to $C_TH^{\frac{\gamma}{2}-\beta-\frac{1}{2}}$. Now we calculate the \textbf{"first nonlinear iteration"}. For $u^{(1)}(t):=\int_{0}^{t}e^{ (t-s)\Lambda^\gamma}\p_x(X^2)\mathrm{d}s$, and let $u_k^{(1)}(t)$ be its Fourier mode
	\begin{align*}
		u_k^{(1)}(t)=N^{2\beta-\gamma+1}\sum_{\substack{l,m\in \mathbb{Z},\vert l\vert\sim\vert m\vert\sim \vert k\vert\sim \vert N\vert\\l+m=k}} \int_0^t e^{-(t-s)\vert k\vert^\gamma}G_{l}(s)G_{m}(s)\mathrm{d}s.
	\end{align*}
	Since the heat integral always provide $N^{-\gamma}$and the sum only has size $N^{1}$, then the inner sum integral has size $N^{\frac12-\gamma}$ with high probability by square root cancellation, then we get the ansatz that
	\begin{equation}\label{eq;ansatz;critical index}
		\beta\leq\frac32\gamma-\frac32.
	\end{equation}
	\par We call a pair $(\beta,\gamma)$ \textbf{``probabilistic admissible pair"} for the Rough or singular Burgers equation on $H^s$ space(In the $\ml C^s$ space, it will be different), if condition \eqref{eq;ansatz;critical index} holds. It is immediate to verify that the pair $(0,\gamma)$ for $\gamma>\frac{4}{3}$ and $(\frac{1}{2},\frac{3}{2})$ are probabilistic admissible pairs; the former regime is treated in the present paper, while the latter will be our future directions.
	\subsection{Future directions}
	In future work, we intend to extend our analysis to more singular regimes of the stochastic Burgers equation. A primary example is the parameter choice 
	$\gamma\geq\frac{3}{2}$ and $\beta=\frac{1}{2}$ in \eqref{eq;Burgers;white noise;singular}. The solution $u$ inherits the spatial regularity of the stochastic convolution $X(t)=\int_0^tP(t-s)\xi(s)\mathrm{d}s$. According to the Lemma \ref{le;regularity}, we have a priori $u\in W^{-\frac{1}{2}-\delta}$ for any $\delta>0$. At such low regularity, the nonlinear term $\p_x(u^2)$ may become ill-defined, as the assumptions required for Lemma \ref{le;est;paraproduct} no longer hold. A standard way to circumvent this difficulty is to work with generalized solutions. More precisely, we consider a regularized solution $u^\epsilon$ satisfying
	\begin{equation}\label{eq;Burgers;app}
		\p_t u^\epsilon-\Lambda^\gamma u^\epsilon=D(u^\epsilon)^2+\vert D\vert^\beta\xi_\epsilon,\quad u^\epsilon(0)=u^\epsilon_0+Y_\epsilon(0),
	\end{equation} 
	and prove that $u^\epsilon$ converges in probability to a limiting process $u$. This approach was originally introduced by Martin Hairer in his analysis of the KPZ equation \cite{hairer2013solving}. In his subsequent foundational work \cite{hairerTheoryRegularityStructures2014}, he developed the comprehensive framework of regularity structures, which provides a systematic way to define solutions for a broad class of singular stochastic partial differential equations—contributions that were recognized with the Fields Medal. Within that framework, the solution is constructed in a space of modelled distributions $\ml D^{\gamma,\eta}_P$ over $[0,T]$ and satisfies the fixed-point equation
	\begin{equation}
		u=(\ml K_{\bar\Gamma}+R_\gamma\ml R)\textbf{R}^+F(u)+Gu_0.
	\end{equation} 
	\appendix
	\section{Estimate for heat flow}
	In this chapter, we give some estimates for heat flow. Defining operator $P(t)$ as
	\begin{equation}\label{eq;linear solution}
		\ml F(P(t)f)=e^{-t\vert \xi\vert^\gamma}\ml Ff,
	\end{equation}
	for $f\in\ml S'(\mathbb{R}^d)$, it's not difficult to see that $P(t)f$ is the solution of
	\begin{equation}
		\p_t P(t)f-\Lambda^\gamma P(t)f=0,\quad P(0)f=f.
	\end{equation}
	For the operator, we firstly claim that it's the linear bounded operator in $L^p$ for $p\geq1$. In fact, setting $\varphi(z)=e^{-\vert z\vert^\gamma}$, then $P(t)$ can be written as $\varphi(t^{\frac{1}{\gamma}}D)$. Noting that $\varphi\in L^1$, it suffice to show $\ml F\varphi\in L^1$. Then for any $t\geq0$, we have 
	\begin{equation}\label{eq;P(t);Bounded}
		\Vert P(t)f\Vert_{L^p}=\Vert \ml F^{-1}\varphi(t^{\frac{1}{\gamma}}\cdot)\ast f\Vert_{L^p}\lesssim \Vert f\Vert_{L^p}.
	\end{equation}
	We can easily get $\ml F\varphi$ belongs to $L^1$ since its the density of a symmetric $\gamma$-stable random variable for $\gamma\in(0,2]$.
	\begin{lemma}\label{le;Schauder estimate;block}
		For $\gamma>0$, $\lambda\in \m Z/\{0\}$, we have the estimate 
		$$\Vert P(t)f\Vert_{L^p}\leq Ce^{-ct\lambda^\gamma}\Vert f\Vert_{L^p},$$
		if the support of $\hat f$ belongs to some $\lambda \ml C$, where $\ml C$ is a fixed annual.
		\begin{proof}
			The proof is similar to Lemma 2.4 in \cite{bahouri2011fourier}, which only need to notice that $\p_x^\beta e^{-t\vert \xi\vert^\gamma}\leq C(1+t)^{\vert \beta\vert}e^{-ct}$ for $\xi\in\lambda \ml C$ will not cause difficult when $\gamma>0$, since $ \vert \xi\vert$ has a uniformly lower bound for $\lambda\in \m Z/\{0\}$.
		\end{proof}
	\end{lemma}
	For Besov space $B_{p,r}^s$, we have the following result.
	\begin{lemma}\label{le;heat flow}
		Let $1\leq p,r\leq\infty$, $\alpha\in \m R$ and $\gamma>0$, then for $t\geq0$, we have the following estimate
		\begin{equation}\label{eq;est;heat flow;1}
			\Vert P(t)f\Vert_{B_{p,r}^\alpha}\lesssim \Vert f\Vert_{B_{p,r}^\alpha}.
		\end{equation}
		Let $\delta\geq0$ and $0<t\leq 1$, we have 
		\begin{equation}\label{eq;est;heat flow;2}
			\Vert P(t)f\Vert_{B_{p,r}^{\alpha+\delta}}\leq t^{-\frac{\delta}{\gamma}}\Vert f\Vert_{B_{p,r}^{\alpha}}.
		\end{equation}
		Moreover, fix $f(0),\ml L f:=(\p_t-\Lambda^\gamma)f\in B_{p,r}^{\alpha}$ , then for $\delta\in[0,\gamma)$ and $t\leq 1$, we have estimate
		\begin{equation}\label{eq;est;heat flow;3}
			\Vert f\Vert_{B_{p,r}^{\alpha+\delta}}\lesssim \Vert f(0)\Vert_{B_{p,r}^{\alpha+\delta}}+t^{1-\frac{\delta}{\gamma}}\Vert \ml L f\Vert_{C_tB_{p,r}^{\alpha}}.
		\end{equation}
		\begin{proof}
			For fixed $j\geq 0$, by Lemma \ref{le;Schauder estimate;block}, for $t>0$, we have
			\begin{align*}
				\Vert P(t)\Delta_jf(t)\Vert_{L^p}\lesssim& Ce^{-ct2^{j\gamma}}\Vert \Delta_jf\Vert_{L^p}\\
				\lesssim& t^{-\frac{\delta}{\gamma}}2^{-j\delta}\Vert \Delta_jf\Vert_{L^p}
			\end{align*}
			For $j=-1$, by \eqref{eq;P(t);Bounded} and $t\leq 1$, we have 
			\begin{align*}
				\Vert P(t)\Delta_{j}f(t)\Vert_{L^p}\lesssim& \Vert \Delta_{j}f(t)\Vert_{L^p}\\
				\lesssim& (2^{-\delta}t^{\frac{\delta}{\gamma}})t^{-\frac{\delta}{\gamma}}2^{\delta}\Vert \Delta_{j}f(t)\Vert_{L^p}\\
				\lesssim &t^{-\frac{\delta}{\gamma}}2^{-j\delta}\Vert \Delta_jf\Vert_{L^p}.
			\end{align*}
			Multiplying $2^{j(\alpha+\delta)}$ both sides and taking the $l^r$-norm with respect to $-1\leq j$, we prove the second result.
			\par For \eqref{eq;est;heat flow;3}, by \eqref{eq;est;heat flow;1} and \eqref{eq;est;heat flow;2}
			\begin{align*}
				\Vert f\Vert_{B_{p,r}^{\alpha+\delta}}\leq& \Vert P(t)f(0)\Vert_{B_{p,r}^{\alpha+\delta}}+\int_0^t\Vert  P(t-s)\ml Lf(s)\Vert_{B_{p,r}^{\alpha+\delta}}\mathrm{d}s\\
				\lesssim&\Vert f(0)\Vert_{B_{p,r}^{\alpha+\delta}}+\int_0^t\Vert (t-s)^{-\frac{\delta}{\gamma}}\ml Lf(s)\Vert_{B_{p,r}^{\alpha}}\mathrm{d}s\\
				\lesssim& \Vert f(0)\Vert_{B_{p,r}^{\alpha+\delta}}+t^{1-\frac{\delta}{\gamma}}\Vert \ml Lf\Vert_{C_tB_{p,r}^{\alpha}}
			\end{align*}
			where we use $\delta<\gamma$.
		\end{proof}
	\end{lemma}
	A similar result carries over to the space $C_T^\alpha L^2$ and $C_T^\alpha L^\infty$, as stated in the following lemma.
	\begin{lemma}\label{le;heat flow;time smoothing}
		Let $\gamma>0$ and $\alpha\in]0,\gamma[$, for any $T\geq0$ we have the following estimate
		\begin{align*}
			\Vert P(\cdot)f\Vert_{\ml W^\alpha_{\infty}(T)}\lesssim \Vert f\Vert_{\ml C^\alpha},\quad 	\Vert P(\cdot)f\Vert_{\ml W^\alpha_{2}(T)}\lesssim \Vert f\Vert_{H^\alpha}.
		\end{align*}
		Moreover, let $\beta\leq 0$, $\delta\in[0,\gamma)$ satisfy $0<\beta+\delta<\gamma$ and $0\leq T\leq 1$, it has the estimate
		\begin{align*}
			\Vert f\Vert_{\ml W_\infty^{\beta+\delta}(T)}\lesssim \Vert f(0)\Vert_{W_\infty ^{\beta+\delta}}+T^{1-\frac{\delta}{\gamma}}\Vert \ml L f\Vert_{C_TW_\infty^{\beta}},\quad \Vert f\Vert_{\ml W_2^{\beta+\delta}(T)}\lesssim \Vert f(0)\Vert_{W_2^{\beta+\delta}}+T^{1-\frac{\delta}{\gamma}}\Vert \ml L f\Vert_{C_TW_2^{\beta}}.
		\end{align*}
		\begin{proof}
			We only proof the case of $\ml W_2^\alpha(T)$. The case of $\ml W_\infty^{\alpha+\delta}(T)$ is the consequence by combining Lemma 2.9 and Lemma 2.11 in \cite{gubinelli2017kpz}. The estimate for $C_TH^{\alpha}$ and $C_TH^{\alpha+\delta}$  is a direct consequence of Lemma \ref{le;heat flow}. For $C_T^{\frac{\alpha}{\gamma}}L^2$, by Parsavel formula, we have
			\begin{align*}
				\Vert P(\cdot)f\Vert_{C_T^{\frac{\alpha}{\gamma}}L^2}=&\sup_{0\leq s<t\leq T}\frac{\Vert (P(t)-P(s))f\Vert_{L^2}}{\vert t-s\vert^{\frac{\alpha}{\gamma}}}\\
				\leq &\sup_{0\leq s<t\leq T}\frac{\Vert e^{-s\vert k\vert^\gamma}(1-e^{-(t-s)\vert k\vert^\gamma})\hat f(k)\Vert_{L^2}}{\vert t-s\vert^{\frac{\alpha}{\gamma}}}
				\leq \Vert f\Vert_{H^{\alpha}},
			\end{align*}
			for $\alpha\in\m R_+$, which we prove the first estimate. Note that
			\begin{align*}
				\Vert f(t)-f(s)\Vert_{L^2}\leq \int_s^t\Vert P(t-\tau)\ml Lf(\tau)\Vert_{L^2}\mathrm{d}\tau+\int_0^s\Vert [P(t-\tau)-P(s-\tau)]\ml Lf(\tau)\Vert_{L^2}\mathrm{d}\tau.
			\end{align*}
			The case of $k=0$ is trivial, we only consider the case of $|k|\geq 1$. The first term can be handled using a method similar to that in Lemma \ref{le;heat flow},  in fact
			\begin{align*}
				\int_s^t\Vert P(t-\tau)\ml Lf(\tau)\Vert_{L^2}\mathrm{d}\tau\lesssim&\Vert \ml Lf\Vert_{C_TH^\beta}\int_s^t\sup_{|k|\geq 1}(e^{-(t-\tau)\vert k\vert^\gamma}\vert k\vert^{-\beta})\mathrm{d}\tau\\
				\lesssim&\Vert \ml Lf\Vert_{C_TH^\beta} (t-s)^{1+\frac{\beta}{\gamma}}.
			\end{align*}
			For the second term, we have estimate
			\begin{align*}
				\int_0^s\Vert [P(t-\tau)-P(s-\tau)]\ml Lf(\tau)\Vert_{L^2}\mathrm{d}\tau\lesssim&\Vert \ml Lf\Vert_{C_TH^\beta}\int_0^s\sup_{|k|\geq 1}(e^{-(s-\tau)\vert k\vert^\gamma}(e^{-(t-s)\vert k\vert^\gamma}-1)\vert k\vert^{-\beta})\mathrm{d}s\\
				\lesssim&\Vert \ml Lf\Vert_{C_TH^\beta}(t-s)^{\frac{\beta+\delta}{\gamma}}s^{1-\frac{\delta}{\gamma}}.
			\end{align*}
			Combining the above estimate, by the definition of $C_T^{\frac{\beta+\delta}{\gamma}}L^2$, we have  
			\begin{align*}
				\Vert f\Vert_{\ml W_2^{\beta+\delta}(T)}=&\sup_{0\leq s<t\leq T}\frac{\Vert f(t)- f(s)\Vert_{L^2}}{\vert t-s\vert^{\frac{\beta+\delta}{\gamma}}}\\
				\leq &T^{1-\frac{\delta}{\gamma}}\Vert \ml Lf\Vert_{C_TH^\beta}.
			\end{align*}
			Then we finish the proof.
		\end{proof}
	\end{lemma}
	\par There are also some similar estimates for $\tilde P(t)$, which is defined as
	$$\ml F(\tilde P(t)(f))=\frac{e^{-t\vert k\vert^\gamma}}{1+\vert k\vert^2}\ml Ff.$$
	To clear that, fix $\varphi(t,z)=\frac{e^{-\vert z\vert^\gamma}}{1+t^{-\frac{2}{\gamma}}\vert z\vert^2}.$ We have $\tilde P(t)f=\varphi(t,t^\frac{1}{\gamma}D)f=\tilde G(t,x)f$. For fixed $t\in[0,T]$, it's easy to see that $\ml F^{-1}\varphi(t,z)\in L^1(\mathbb{R})$. In fact, by Young's inequality and \eqref{eq;linear solution}
	\begin{align*}
		\Vert \ml F^{-1}_z\varphi(t,z)(\xi)\Vert_{L^1}=&\Vert\ml F^{-1}{e^{-\vert z\vert^\gamma}}\ast\ml F(\frac{1}{1+t^{-\frac{2}{\gamma}}\vert z\vert^2}) \Vert_{L^1}\\
		\leq& \Vert \check{g}\Vert_{L^1}\Vert \ml F(\frac{1}{1+t^{-\frac{2}{\gamma}}\vert z\vert^2})\Vert_{L^1}
	\end{align*}
	the first integral is always finite. For the second term, by $\ml F^{-1}_z{\frac{1}{1+\vert z\vert^2}}=\pi e^{-\vert \xi\vert}$, we have $\ml F(\frac{1}{1+t^{-\frac{2}{\gamma}}\vert z\vert^2})=\pi t^\frac{1}{\gamma}e^{-t^\frac{1}{\gamma}\vert \xi\vert}$ which also in $L^1$, 
	which obtain the following claim
	\begin{align*}
		\Vert \tilde P(t)f\Vert_{B_{p,r}^\alpha}\lesssim \Vert f\Vert_{B_{p,r}^\alpha}.
	\end{align*}
	On the other hand, for $j\geq 1$, we have the estimate
	\begin{align*}
		\Vert (\ml F^{-1}\varphi(t,t^\frac{1}{\gamma}z))\ast\Delta_j u\Vert_{L^p}\lesssim&\Vert \ml F^{-1}(\varphi(t,t^\frac{1}{\gamma}z)\psi(2^{-j}z))\Vert_{L^1}\Vert \Delta_ju\Vert_{L^p}\\
		\lesssim&\Vert \ml F^{-1}(\varphi(t,t^\frac{1}{\gamma}2^jz)\psi)\Vert_{L^1}\Vert \Delta_ju\Vert_{L^p}\\
		\lesssim&\Vert (1+\vert x\vert)^{d+1} \ml F^{-1}(\varphi(t,t^\frac{1}{\gamma}2^jz)\psi)\Vert_{L^\infty}\Vert \Delta_ju\Vert_{L^p}\\
		\lesssim&\Vert (1+\vert D\vert)^{d+1}\varphi(t,t^\frac{1}{\gamma}2^jz)\psi\Vert_{L^1}\Vert \Delta_ju\Vert_{L^p}\\
		\lesssim&(1+2^jt^\frac{1}{\gamma})^{d+1}\max_{\mu\in \mathbb{N}^d:\vert \mu\vert\leq d+1}\Vert \partial^\mu\varphi(t,t^\frac{1}{\gamma}2^j\cdot)\Vert_{L^\infty(\mathrm{supp}(\psi))}\Vert \Delta_j u\Vert_{L^p}.
	\end{align*}
	By Lebnitz'rule $\partial^2(uv)=\sum_{\mu_1+\mu_2=2}\frac{\alpha!}{\mu_1!\mu_2!}\partial^{\mu_1}u\partial^{\mu_2}v$, we have
	$$\partial^2 \varphi(t,z)=\sum_{\mu_1+\mu_2=2}\frac{\alpha!}{\mu_1!\mu_2!}\partial^{\mu_1}(e^{-\vert z\vert^\gamma})\partial^{\mu_2}(\frac{1}{1+t^{-\frac{2}{\gamma}}\vert z\vert^2})=c_1u''v+c_2u'v'+c_3uv'',$$
	where $u(z)=e^{-\vert z\vert^\gamma}$ and $v(z),v'(z),v''(z)$ satisfied
	$$v:=\frac{1}{1+cz^2},\quad v':=\frac{-2cz}{(1+cz^2)^2},\quad v''=\frac{6c^2z^2-2c}{(1+cz^2)^3}.$$
	
	Since $cz^2>0$, it's obviously that $\vert v\vert,\vert v'\vert,\vert v''\vert\lesssim 1$ which shows that 
	\begin{align*}
		\sup_{\vert \mu\vert\leq 2}\sup_{z\geq1}(1+\vert z\vert)^{\delta+2}\vert \partial_z^\mu\varphi(t,z)\vert\lesssim 1.
	\end{align*}
	On the other hand, since all partial derivatives of $u=e^{-\vert z\vert^\gamma}$ is decay faster than rational function for $\gamma>0$, we have
	\begin{align*}
		\sup_{\vert \mu\vert\leq 2}\sup_{z\geq1}(1+\vert z\vert)^{\delta+2}\vert \partial_z^\mu\varphi(t,z)\vert\lesssim &\sup_{z\geq1}(1+\vert z\vert)^{\delta+2d}\sum_{\vert \mu\vert\leq 2}\vert \p_z^\mu\varphi(t,z)\vert\\
		\lesssim &c^{-1}\sup_{z\geq1}(1+\vert z\vert)^{\delta}\vert (u''+u'+u)\frac{cz^2}{1+c z^2}\vert\\
		&+c^{-1}\sup_{z\geq1}(1+\vert z\vert)^{\delta-1}\vert (u'+u)\frac{-2c^2z^4}{(1+cz^2)^2}\vert\\
		&+c^{-1}\sup_{z\geq1}(1+\vert z\vert)^{\delta-2}\vert u\frac{6c^3z^6-2c^2z^4}{(1+cx^2)^3}\vert\\
		\lesssim&c^{-1}.
	\end{align*}
	Let $c=t^{-\frac{2}{\gamma}}$ and $d=1$, then there exists $j_0$ satisfying that $2^{j_0}t^{\frac{1}{\gamma}}\geq 1$ such that
	\begin{align*}
		\Vert \tilde P(t)\Delta_jf(t)\Vert_{L^p}\lesssim&(1+2^jt^{\frac{1}{\gamma}})^2(1+2^jt^{\frac{1}{\gamma}})^{-2-\delta}\min (t^{\frac{2}{\gamma}},1)\\
		\lesssim&\min (t^{\frac{2}{\gamma}},1)t^{-\frac{\delta}{\gamma}}2^{-j\delta}\Vert \Delta_jf(t)\Vert_{L^p}.
	\end{align*}
	For $-1\leq j<j_0$, we have the estimate
	\begin{align*}
		\Vert \tilde P(t)\Delta_{j}f(t)\Vert_{L^p}\lesssim& \Vert \Delta_{j}f(t)\Vert_{L^p}\\
		\lesssim& (\max (t^{-\frac{2}{\gamma}},1)2^{j_0\delta}t^{\frac{\delta}{\gamma}})\min (t^{\frac{2}{\gamma}},1)t^{-\frac{\delta}{\gamma}}2^{-j\delta}\Vert \Delta_{j}f(t)\Vert_{L^p}\\
		\lesssim &\min (t^{\frac{2}{\gamma}},1)t^{-\frac{\delta}{\gamma}}2^{-j\delta}\Vert \Delta_jf(t)\Vert_{L^p}.
	\end{align*}
	Combining the proceeding results, we obtain the following lemma
	\begin{corollary}\label{coro;heat flow}
		Let $\gamma>0$, $\tilde P(t)$ satisfy $\ml F(\tilde P(t)(f))=\frac{e^{-t\vert \xi\vert^\gamma}}{1+\vert \xi\vert^2}\ml Ff$. Fixing some $T>0$, then for any $t\in(0,T]$, $\alpha\in\mathbb{R}$, $\delta\in(0,\gamma)$, and $u\in S'(\mathbb{R})$, we have the following estimate
		\begin{equation}\label{eq;est;heat flow}
			\Vert \tilde{P}(t)u\Vert_{B^{\alpha+\delta}_{p,r}}\lesssim \min(1,t^\frac{2}{\gamma})t^{-\frac{\delta}{\gamma}}\Vert u\Vert_{B^{\alpha}_{p,r}}.
		\end{equation}
		Moreover, let $\tilde Lf=(\p_t-\Lambda^\gamma)(1-\p_x^2)^{-1}f\in B_{p,r}^\alpha$, then for $\delta\in[0,\gamma+2)$ and $t\leq 1$, we have the following estimate
		\begin{equation}
			\Vert f\Vert_{B_{p,r}^{\alpha+\delta}}\leq \Vert f(0)\Vert_{B_{p,r}^{\alpha+\delta}}+Ct^{1-\frac{\delta-2}{\gamma}}\Vert \ml L f\Vert_{B_{p,r}^{\alpha}}.
		\end{equation}
	\end{corollary}  
	\section{Commutator estimates}\label{sec;commutator est}
	First we introduce a useful commutator estimates as Lemma 2.99 in  \cite{bahouri2011fourier}.
	\begin{lemma}\label{le;commutater;D}[\cite{bahouri2011fourier}]
		Let $f$ be a smooth function on $\mathbb{R}^d$ and be homogeneous of degree $m$ away from $0$. Let $\rho\in(0,1)$, $s\in\mathbb{R}$ and $p,r \in[1,\infty]$. If $p_1,p_2\in[1,\infty]$ such that $\frac{1}{p_1}+\frac{1}{p_2}=\frac{1}{p}$, then the bound
		\begin{align*}
			\Vert a\prec f(D)u-f(D)(a\prec u)\Vert_{B_{p,r}^{s-m+\rho}}\lesssim \Vert \nabla a\Vert_{B_{p_1,\infty}^{\rho-1}}\Vert u\Vert_{B_{p_2,r}^s}.
		\end{align*}
		is hold.
	\end{lemma}
	\begin{remark}
		Fixing $f(x)=\vert x\vert^\gamma$, then we have $\Lambda^\gamma(u\prec v)-u\prec \Lambda^\gamma v\in W_p^{\alpha+\beta-\gamma}$ for $u\in\ml C^\alpha$, $\alpha\in(0,1)$ and $v\in W_p^\beta$ for $\alpha\in(0,1)$.
	\end{remark}
	\begin{lemma}\label{le;commutator;resonant}
		Assume that $\alpha\in(0,1)$, $\beta,\gamma\in \mathbb{R}$ and satisfies $\beta+\gamma<0$, $\alpha+\beta+\gamma>0$, let $C(f,g,h)=(f\prec g)\circ h-f(g\circ h)$, then we have the following estimate
		$$\Vert C(f,g,h)\Vert_{H^{\alpha+\beta+\gamma}}\leq \Vert f\Vert_{\ml C^\alpha}\Vert g\Vert_{H^\beta}\Vert v\Vert_{\ml C^\gamma}.$$
		\begin{proof}
			We rewrite 
			\begin{align*}
				C(f,g,h)=\sum_{j,k\geq -1}\sum_{\vert i-j\vert\leq 1}[1_{i\gtrsim k}R_i(\Delta_k f,g)\Delta_j h-1_{i\leq k-N}\Delta_k f\Delta_jg\Delta_jh]	.
			\end{align*}
			It's suffice to estimate above two terms. For any fixed $k$, we have estimate 
			\begin{align*}
				\Vert \sum_{j\geq -1}\sum_{\vert i-j\vert\leq 1}1_{i\leq k-N}\Delta_k f\Delta_jg\Delta_jh\Vert_{L^2}\lesssim&2^{-k\alpha}\Vert f\Vert_{\ml C^\alpha}\sum_{i=-1}^{k-N}2^{-i(\beta+\gamma)}\Vert g\Vert_{H^\beta}\Vert h\Vert_{\ml C^\gamma}\\
				\lesssim&2^{-k(\alpha+\beta+\gamma)}\Vert f\Vert_{\ml C^\alpha}\Vert g\Vert_{H^\beta}\Vert h\Vert_{\ml C^\gamma},
			\end{align*}
			where we use the fact $\beta+\gamma<0$. Note that the Fourier transform of $\sum_{j,k\geq -1}\sum_{\vert i-j\vert\leq 1}1_{i\gtrsim k}R_i(\Delta_k f,g)\Delta_j h$ is supported in a ball $2^k\ml B$, by Lemma 2.84 in \cite{bahouri2011fourier}, we obtain the estimate since $\alpha+\beta+\gamma>0$. For the first series, we can check that the Fourier transform of $\sum_{k\geq -1}\sum_{\vert i-j\vert\leq 1}1_{i\gtrsim k}R_i(\Delta_k f,g)\Delta_j h$ is supported in a ball $2^j\ml B$ for fixed $j$. And we calculate that
			\begin{align*}
				\Vert \sum_{k\geq -1}\sum_{\vert i-j\vert\leq 1}1_{i\gtrsim k}R_i(\Delta_k f,g)\Delta_j h\Vert_{L^2}=&\Vert \sum_{\vert i-j\vert\leq 1}R_i(\sum_{1\leq k\lesssim i}\Delta_k f,g)\Delta_j h\Vert_{L^2}\\
				\lesssim &\sum_{\vert i-j\vert\leq 1}2^{-i(\beta+\gamma)}\Vert \sum_{k\lesssim i}\Delta_k f\Vert_{L^\infty} \Vert g\Vert_{H^\beta} 2^{-j\gamma}\Vert h\Vert_{\ml C^\gamma}\\
				\lesssim &2^{-j(\alpha+\beta+\gamma)}\Vert f\Vert_{\ml C^\alpha}\Vert g\Vert_{H^\beta}\Vert h\Vert_{\ml C^\gamma}.
			\end{align*}
			Combining all estimates above, we finish our proof.
		\end{proof}
	\end{lemma}
	\begin{lemma}\label{le;commutator;modfied paraproduct, L}
		Let $\alpha\in(0,\gamma)$ and $\beta\in\mathbb{R}$. Then for $t>0$, we have the following estimate
		\begin{equation}\label{eq;commutator;modfied paraproduct, L;1}
			\Vert f\pprec g-f\prec g\Vert_{W^{\alpha+\beta}}\lesssim \Vert f\Vert_{\ml W_\infty^\alpha(t)}\Vert g\Vert_{W^\beta}.
		\end{equation}
		Moreover, define operator $\ml L=\p_t-\Lambda^\gamma$, then for $\gamma>0$ and $\alpha\in(0,1)$, we have the following estimate
		\begin{equation}\label{eq;commutator;modfied paraproduct, L;2}
			\Vert \ml L(f\pprec g)(t)-(f\pprec\ml Lg)(t)\Vert_{W^{\alpha+\beta-\gamma}}\lesssim \Vert f\Vert_{\ml W_\infty^\alpha(t)}\Vert g\Vert_{W^\beta}.
		\end{equation}
		\begin{proof}
			We only proof the case of $H^\alpha$, the case of $\ml C^\alpha$ is similar. Firstly, control
			\begin{align*}
				\Vert (\int_{\m R}2^{\gamma j}\varphi(2^{\gamma j}(t-s))S_{j-1}f(s)\mathrm{d}s-S_{j-1}f(t))\Delta_j g(t)\Vert_{L^2}.
			\end{align*}
			By $\int_{\m R}\varphi(s)\mathrm{d}s=1$, we have 
			\begin{align*}
				\Vert (\int_{\m R}2^{\gamma j}\varphi(2^{\gamma j}(t-s))S_{j-1}f(s)\mathrm{d}s-&S_{j-1}f(t))\Delta_j g(t)\Vert_{L^2}\\
				= &\Vert \int_{\m R}2^{\gamma j}\varphi(2^{2j}s)S_{j-1}f(t-s)-S_{j-1}f(t)\mathrm{d}s\Delta_j g(t)\Vert_{L^2}\\
				\lesssim&\Vert \Delta_j g(t)\Vert_{L^2}\int_{\m R}2^{\gamma j}\varphi(2^{\gamma j}s)\Vert S_{j-1}f(t-s)-S_{j-1}f(t)\Vert_{L^\infty}\mathrm{d}s\\
				\lesssim&\Vert \Delta_j g(t)\Vert_{L^2}\Vert f(t)\Vert_{C_t^\frac{\alpha}{\gamma} L^\infty}\int_{\m R}2^{\gamma j}\varphi(2^{\gamma j}s)\vert s\vert^\frac{\alpha}{\gamma} \mathrm{d}s\\
				\lesssim&2^{-j(\alpha+\beta)}(2^{j\beta}\Vert \Delta_j g(t)\Vert_{L^2})\Vert f(t)\Vert_{\ml W_\infty^\alpha(t)}.
			\end{align*}
			Multiplying $2^{-j(\alpha+\beta)}$ both sides to the above inequality and summing over $j$ in the $l^2$, we prove \eqref{eq;commutator;modfied paraproduct, L;1}.
			\par For the second result, we firstly control the term of fractional Laplacian commutator
			\begin{align*}
				\Vert \Lambda^\gamma (f\pprec g)-f\pprec \Lambda^\gamma g\Vert_{C_TW^{\alpha+\beta-\gamma}}
			\end{align*}
			For $f\in C_t\ml C^\alpha$, rewriting $\Lambda^\gamma(f\pprec g)(t)-f\pprec \Lambda^\gamma g(t)$ as
			\begin{align*}
				\sum_{j\geq 1}\int_0^t2^{\gamma j}&\varphi(2^{\gamma j}(t-s))[\Lambda^\gamma(S_{j-1}f(s)\Delta_j g(t))-S_{j-1}f(s)\Lambda^\gamma(\Delta_j g(t))]\mathrm{d}s\\
				=&\sum_{j\geq 1}\int_0^t2^{\gamma j}\varphi(2^{\gamma j}(t-s))[S_{j-1}f,\Lambda^\gamma \widetilde{\Delta}_j]\Delta_j g(t)\mathrm{d}s
			\end{align*}
			where $\tilde \rho$ is a smooth function supported in an annuals and with value 1 on a neighborhood of $\mathrm{Supp} \rho+\mathrm{Supp}\chi(\cdot/4)$. Observing that there exists $N_0$ such that
			\begin{align*}
				\forall j\geq N_0,\quad \Lambda^\gamma \widetilde{\Delta}_j=2^{j\gamma} (\vert \cdot\vert^\gamma \tilde\rho)(2^{-j}D).
			\end{align*}
			Then by Lemma 2.97 in \cite{bahouri2011fourier}, we have for any $j\geq N_0$
			\begin{align*}
				\Vert [S_{j-1}f,\Lambda^\gamma \widetilde{\Delta}_j]\Delta_j g(t)\Vert_{L^p}\leq C2^{j(\gamma-1)} \Vert S_{j-1}f(s)\Vert_{L^\infty} \Vert \Delta_j g(t)\Vert_{L^p},
			\end{align*}
			and for $1\leq j<N_0$, we have 
			\begin{align*}
				\Vert [S_{j-1}f,\Lambda^\gamma \widetilde{\Delta}_j]\Delta_j g(t)\Vert_{L^p}\leq C2^{j(\gamma-1)} \Vert \nabla S_{j-1}f(s)\Vert_{L^\infty} \Vert \Delta_j g(t)\Vert_{L^p}
			\end{align*}
			Since $\Vert \nabla S_{j-1} f(s)\Vert_{L^\infty} \leq C2^{j(1-\alpha)} \Vert \nabla f(s)\Vert_{\ml C^{\alpha-1}}\leq C2^{j(1-\alpha)} \Vert  f\Vert_{C_t\ml C^{\alpha}}$ for $\alpha<1$ and $0<s<t$. then we have
			\begin{align*}
				2^{j(\beta+\alpha-\gamma)}\Vert \int_0^t2^{\gamma j}\varphi(2^{\gamma j}(t-s))[S_{j-1}f,\Lambda^\gamma \widetilde{\Delta}_j]\Delta_j g(t)\mathrm{d}s\Vert_{L^p}\lesssim \Vert  f\Vert_{C_t\ml C^{\alpha}}2^{j\beta}\Vert \Delta_j g(t)\Vert_{L^p}
			\end{align*}
			since $[S_{j-1}f,\Lambda^\gamma \widetilde{\Delta}_j]$ is spectrally supported in dyadic annuli, by Lemma 2.23 in \cite{bahouri2011fourier}, we obtain the estimate.
			\par For $\p_t (f\pprec g)-f\pprec \p_tg$, it's suffice to estimate
			$$\sum_j\p_t(\int_0^t2^{\gamma j}\varphi(2^{\gamma j}(t-s))S_{j-1}f(s)\mathrm{d}s)\Delta_jg(t).$$ 	
			For fixed $j$ we recall that $\mathrm{supp}(\varphi)\subset \mathbb{R}_+$, and therefore 
			$$\int_0^t2^{2\gamma j}\varphi'(2^{\gamma j}(t-s))S_{j-1}f(s)\mathrm{d}s=2^{\gamma j}\int_{\mathbb{R}}2^{\gamma j}\varphi'(2^{\gamma j}(t-s))S_{j-1}f(s)1_{s\geq0}\mathrm{d}s.$$
			Since $\varphi'(0)=\varphi(0)=0$, \eqref{eq;commutator;modfied paraproduct, L;2} then follows as the first part of the proof.
		\end{proof}
	\end{lemma}
	\textbf{Acknowledgments} This work was partially supported by the National Natural Science Foundation of China (No.12571261).
	\bibliographystyle{alpha} 
	\bibliography{ref}
\end{document}